\documentclass[12pt]{article}

\usepackage{setspace}
\usepackage{amsmath,amsthm,amsfonts,amssymb,bm,wasysym}
\usepackage[usenames]{color}
\usepackage{verbatim}
\usepackage{xcolor}
\usepackage[hyperfootnotes=false]{hyperref}
\hypersetup{
 colorlinks,
 citecolor=blue,
 linkcolor=red,
 urlcolor=blue}

\parskip \medskipamount
\usepackage{mathtools}
\DeclarePairedDelimiter\ceil{\lceil}{\rceil}
\DeclarePairedDelimiter\floor{\lfloor}{\rfloor}

\newtheorem{theorem}{Theorem}[section]
\newtheorem{definition}[theorem]{Definition}
\newtheorem{assumption}[theorem]{Assumption}
\numberwithin{equation}{section}
\newtheorem{lemma}[theorem]{Lemma}
\newtheorem{proposition}[theorem]{Proposition}
\newtheorem{corollary}[theorem]{Corollary}

\newtheorem{remark}[theorem]{Remark}

\newtheorem{conjecture}[theorem]{Conjecture}
\numberwithin{equation}{section}

\newtheorem*{theorem-non}{Theorem}

\def\bP{\mathbb{P}}

\renewcommand{\phi}{\varphi}
\renewcommand{\epsilon}{\varepsilon}

\newcommand{\Ex}{\mathop{\mathbb{E}}}  
\newcommand{\argmax}{\mathop{\text{argmax}}}  

\newcommand{\E}[1]{\mathbb{E}\!\left[#1\right]}

\newcommand{\tn}{|\kern-.1em|\kern-0.1em|}

\mathchardef\ordinarycolon\mathcode`\:
\mathcode`\:=\string"8000
\begingroup \catcode`\:=\active
\gdef:{\mathrel{\mathop\ordinarycolon}}
\endgroup

\newcommand{\inner}[1]{\langle #1 \rangle}

\begin{document}
\title{Free Energy Wells and Overlap Gap Property\\ in Sparse PCA}
\author{
{\sf G\'erard Ben Arous}\thanks{NYU; e-mail: {\tt benarous@cims.nyu.edu}}
\and
{\sf Alexander S.\ Wein}\thanks{NYU; e-mail: {\tt awein@cims.nyu.edu}. Partially supported by NSF grant DMS-1712730 and by the Simons Collaboration on Algorithms and Geometry.}
\and
{\sf Ilias Zadik}\thanks{NYU; email: {\tt zadik@nyu.edu}. Supported by a CDS Moore-Sloan Postdoctoral Fellowship. }
}
\maketitle

\begin{abstract}
We study a variant of the sparse PCA (principal component analysis) problem in the ``hard'' regime, where the inference task is possible yet no polynomial-time algorithm is known to exist. Prior work, based on the low-degree likelihood ratio, has conjectured a precise expression for the best possible (sub-exponential) runtime throughout the hard regime. Following instead a statistical physics inspired point of view, we show bounds on the depth of free energy wells for various Gibbs measures naturally associated to the problem. These free energy wells imply hitting time lower bounds that corroborate the low-degree conjecture: we show that a class of natural MCMC (Markov chain Monte Carlo) methods (with worst-case initialization) cannot solve sparse PCA with less than the conjectured runtime. These lower bounds apply to a wide range of values for two tuning parameters: temperature and sparsity misparametrization. Finally, we prove that the Overlap Gap Property (OGP), a structural property that implies failure of certain local search algorithms, holds in a significant part of the hard regime.\footnote{Accepted for presentation at the Conference on Learning Theory (COLT) 2020.}
\end{abstract}

\newpage
\tableofcontents
\newpage

\section{Introduction}

\subsection{The Model}\label{sec:model}

We consider the following variant of sparse PCA in the spiked Wigner model (also called \emph{principal submatrix recovery}). Let $W$ be a $\mathrm{GOE}(n)$ matrix, i.e., $n \times n$ symmetric with off-diagonal entries $\mathcal{N}(0,1/n)$ and diagonal entries $\mathcal{N}(0,2/n)$, all independent aside from the symmetry $W_{ij} = W_{ji}$. Let $x$ be an unknown $k$-sparse vector in $\{0,1\}^n$, i.e., exactly $k$ entries are equal to $1$. We are interested in recovering $x$ from the observation \[ Y=\frac{\lambda}{k} xx^\top+W \]
where $\lambda > 0$ is the \emph{signal-to-noise ratio}. We study the problem in the limit $n \to \infty$, where the parameters $\lambda = \lambda_n$ and $k = k_n$ may depend on $n$. We are primarily interested in the \emph{exact recovery} problem: we study algorithms which given $Y$, output $x$ \emph{with high probability}, i.e., probability tending to $1$ as $n \to \infty$. Our regime of interest will be $1 \ll k \ll n$. Throughout, we use the notation $\ll$ to hide factors of $n^{o(1)}$ (although in most cases, $\ll$ will only hide logarithmic factors).

\subsection{Our Contributions}\label{sec:contrib}

Prior work (which we review in detail in Section~\ref{sec:prior}) suggests the existence of a ``hard regime'' $\sqrt{k/n} \ll \lambda \ll \min\{1,k/\sqrt{n}\}$ where exact recovery is information-theoretically possible but no polynomial-time algorithm is known. More specifically, the work of \cite{ding-subexp} suggests the following conjecture regarding a precise expression for the best possible (sub-exponential) runtime throughout the hard regime.

\begin{conjecture}\label{conj:main}
Consider the sparse PCA problem as defined in Section~\ref{sec:model}. For any $\lambda$ in the ``hard'' regime $\sqrt{k/n} \ll \lambda \ll \min\{1,k/\sqrt{n}\}$, any algorithm for exact recovery requires runtime $\exp\left(\tilde\Omega\left(\frac{k^2}{\lambda^2 n}\right)\right)$.
\end{conjecture}
This prediction is made by \cite{ding-subexp} (for a variant of our model where $x_i \in \{0,-1,1\}$) using the \emph{low-degree likelihood ratio} \cite{HS-bayesian,sos-detecting,hopkins-thesis}, which amounts to studying the power of algorithms based on low-degree polynomials.
There are known algorithms which achieve the matching runtime $\exp\left(\tilde O\left(\frac{k^2}{\lambda^2 n}\right)\right)$ \cite{ding-subexp,anytime-pca}. For instance, the following simple algorithm of~\cite{ding-subexp} proceeds in two steps. The first step is to let $k' \approx \frac{k^2}{\lambda^2 n}$ and solve, by exhaustive search, the optimization problem
\begin{equation}\label{eq:argmax}
\argmax_{v \in S_{k'}} v^\top Y v
\end{equation}
where $S_{k'}$ is the space of $k'$-sparse vectors
\begin{equation}\label{eq:Sk}
S_{k'} = \{v \in \{0,1\}^n \,:\, \|v\|_0 = k'\},
\end{equation}
and the final step uses the optimizer $v^*$ to exactly recover $x$ via a simple boosting procedure (discussed in Section~\ref{sec:boost}).

In this paper we give evidence in support of Conjecture~\ref{conj:main} by showing the existence of \emph{free energy wells} in the Gibbs measure (at various temperatures) associated with the optimization problem \eqref{eq:argmax} for various choices of the tuning parameter $k'$. As explained in Proposition \ref{prop:mcmc}, a free energy well (defined in Section~\ref{sec:FEW}) of depth $D$ at inverse temperature $\beta$ implies that a certain class of MCMC methods with parameter $\beta$ requires time at least $\exp(\Omega(D))$ to solve \eqref{eq:argmax}. Our main result can be stated informally as follows.

\begin{theorem-non}\emph{(Main result, informal)}
Suppose $\lambda$ is in the ``hard'' regime $\sqrt{k/n} \ll \lambda \ll \min\{1,k/\sqrt{n}\}$ and that additionally, $\lambda \ll (k/n)^{1/4}$. For any ``informative'' $k'$ and any $\beta \ge 0$ (possibly depending on $n$), there exists a free energy well of depth $\tilde\Omega\left(\frac{k^2}{\lambda^2 n}\right)$ with high probability.
\end{theorem-non}
Here ``informative'' $k'$ refers to the condition~\eqref{eq:k-good} (roughly $\frac{k^2}{\lambda^2 n} \lesssim k' \lesssim \lambda^2 n$) which captures the $k'$ values for which solving the optimization problem \eqref{eq:argmax} is actually useful in the sense that a near-optimal solution can be used to exactly recover $x$ via a simple boosting procedure (see Section~\ref{sec:boost}). Our main result shows that if the condition $\lambda \ll (k/n)^{1/4}$ is satisfied then MCMC cannot improve the runtime of \cite{ding-subexp,anytime-pca} for \emph{any} choice of inverse temperature $\beta$ and \emph{any} (informative) choice of misparametrization $k'$. The main weakness of the result is the condition $\lambda \ll (k/n)^{1/4}$, which is an artifact of the proof. However, in the relatively sparse regime $k \ll n^{1/3}$, the condition $\lambda \ll (k/n)^{1/4}$ holds throughout the entire ``hard'' regime. Thus we obtain a complete refutation of MCMC methods (across all $\beta$ and $k'$) throughout a large range of sparsity values (namely $k \ll n^{1/3}$).

Our results are actually somewhat stronger than what we have stated here: even when the condition $\lambda \ll (k/n)^{1/4}$ does not hold, the result still holds for some $k'$ values; in particular, it always holds for all informative $k' \le k$. One consequence of this is that it is not possible to speed up the algorithm of \cite{ding-subexp} by taking their choice of $k' \approx \frac{k^2}{\lambda^2 n}$ (the smallest ``informative'' $k'$, which in particular is less than $k$) and solving \eqref{eq:argmax} via MCMC instead of exhaustive search.

\begin{remark}\label{rem:tight}
In order for a computational hardness result to be most compelling, the class of algorithms ruled out should capture the best known algorithms. This is indeed the case here in the sense that there exists a choice of parameters, namely $k' \approx \frac{k^2}{\lambda^2 n}$ and $\beta = 0$, for which MCMC (followed by boosting) mimics the algorithm of \cite{ding-subexp} and achieves the runtime $\exp\left(\tilde O\left(\frac{k^2}{\lambda^2 n}\right)\right)$. For this choice of parameters, MCMC is simply a random walk (ignoring the data $Y$) on the space of $k'$-sparse vectors, which will visit all states within time $\exp(\tilde O(k')) = \exp\left(\tilde O\left(\frac{k^2}{\lambda^2 n}\right)\right)$ with high probability; see Appendix~\ref{app:pf-random-walk} for a proof. A consequence is that for this choice of $k'$ and $\beta$, any free energy well has depth $\tilde O\left(\frac{k^2}{\lambda^2 n}\right)$ and so our lower bound is tight. It is not clear whether MCMC with more natural parameters (e.g.\ $k' = k$) matches the above runtime; it might in fact be strictly worse. This highlights the importance of allowing $k' \ne k$ in our main result.
\end{remark}

\subsection{Prior Work: Algorithms and Lower Bounds}\label{sec:prior}

We now review some prior work that has proposed and analyzed various algorithms for our variant of sparse PCA.

\begin{itemize}

    \item {\bf PCA}: If $\lambda$ is fixed (not depending on $n$), it is well known in random matrix theory that the leading eigenvalue and eigenvector of $Y$ undergo a sharp phase transition at $\lambda = 1$. Namely, if $\lambda > 1$ then the leading eigenvector $v_1(Y)$ achieves \emph{weak recovery}: \[ \frac{\langle v_1(Y), x \rangle^2}{\|v_1(Y)\|^2 \|x\|^2} \stackrel{a.s.}{\to} 1 - \lambda^{-2} > 0 \] \cite{BBP,FP-wigner,wigner-eigenvec}. Importantly, this method does not exploit the fact that $x$ is sparse. However, the sparsity can be used to boost weak recovery to exact recovery (see Section~\ref{sec:boost}).

    \item {\bf MLE}: The maximum likelihood estimator is $\text{argmax}_{v \in S_k} v^\top Y v$ where $S_k = \{v \in \{0,1\}^n \,:\, \|v\|_0 = k\}$. This method achieves exact recovery provided $\lambda \gg \sqrt{k/n}$ (see \cite{banks-sparse}). However, computing this by exhaustive search over $S_k$ has runtime $\exp(\tilde O(k))$ since $|S_k| = {n \choose k}$.

    \item {\bf Diagonal thresholding}: The simple \emph{diagonal thresholding} algorithm \cite{JL-2,AW-sparse} identifies the $k$ largest diagonal entries of $Y$ and reports these indices as the support of $x$. This achieves exact recovery provided $\lambda \gg k/\sqrt{n}$.
    
    \item {\bf Subexponential-time algorithms}: From above we have polynomial-time exact recovery when $\lambda \gg \min\{1,k/\sqrt{n}\}$. No poly-time algorithm is known when $\lambda \ll \min\{1,k/\sqrt{n}\}$ (even for weak recovery), suggesting a ``possible but hard'' regime when $\sqrt{k/n} \ll \lambda \ll \min\{1,k/\sqrt{n}\}$. The existence of such a regime is known as a \emph{statistical-to-computational gap}. Precise runtime estimates throughout the hard regime were studied by \cite{ding-subexp,anytime-pca}, giving an algorithm of subexponential runtime $\exp\left(\tilde O\left(\frac{k^2}{\lambda^2 n}\right)\right)$. This improves upon the runtime of the MLE and gives a smooth tradeoff between runtime and the signal-to-noise ratio $\lambda$.
\end{itemize}

\noindent For the related \emph{detection} problem of hypothesis testing between $Y = \frac{\lambda}{k} xx^\top + W$ and pure noise $Y = W$, a simple ``sum test'' (sum all entries of $Y$) succeeds with high probability provided $\lambda = \omega(\sqrt{n}/k)$ (see e.g., \cite{BBH-reduction}). This observation suggests that poly-time detection is easier than poly-time recovery when $k \gg \sqrt{n}$. See \cite{detect-submatrix} for a precise analysis of the detection problem.

The proportional regime $k = \Theta(n)$ (which we do not consider in this paper) has also received attention \cite{amp-sparse,LKZ-sparse,LKZ-lowrank,OGP-submatrix}. Here an AMP (approximate message passing) algorithm gives polynomial-time Bayes-optimal recovery provided $k/n$ exceeds a certain constant \cite{amp-sparse}. Related to the current paper, \cite{OGP-submatrix} studies OGP and low-temperature MCMC methods in the proportional regime. In contrast, the techniques we use here are more elementary (although quite technically involved), using first- and second-moment arguments rather than appealing to sophisticated machinery based on the Parisi formula.

Many variants of sparse PCA, other than the one we study in this paper, have been considered in the literature. Often, the \emph{spiked covariance (Wishart) model} \cite{johnstone-spiked,JL-1,BBP} is used in place of the spiked Wigner model. Also, various assumptions on the structure of $x$ can be made, e.g., the nonzero entries can be $\pm 1$ or unconstrained. Many algorithms have been proposed and analyzed in these related settings; see e.g., \cite{covariance-thresh} and references therein.

\paragraph{Lower bounds.} Tight information-theoretic lower bounds are known for variants of sparse PCA \cite{PJ-sparse,VL-minimax,BR-opt,CMW-opt}, showing that the MLE is essentially optimal in the minimax sense. In the regime $k = \Theta(n)$, sharp information-theoretic thresholds are given by the replica formula from statistical physics; see e.g., the survey \cite{miolane-survey} and references therein. To the best of our knowledge, information-theoretic lower bounds have not been shown for our precise setting of interest. However (as is typical in these types of sparse models) we expect that the MLE is essentially information-theoretically optimal, i.e., weak recovery is impossible when $\lambda \ll \sqrt{k/n}$. This has been shown when $k \ge n^{1-\beta}$ for $\beta $ smaller than a certain constant \cite{barbier-zero-one}. It is at least easy to see that \emph{exact} recovery is impossible when $\lambda \ll \sqrt{k/n}$, since even distinguishing between a fixed pair of adjacent signals (differing in only 2 coordinates) is impossible.

In this paper we study the conjectured hard regime in our variant of sparse PCA. Many other variants of sparse PCA are also conjectured to exhibit similar statistical-to-computational gaps and these hard regimes have attracted much recent attention. Evidence for computational hardness in the conjectured hard regime has been given, including reductions from planted clique \cite{BR-reduction,WBS-reduction,BBH-reduction,BB-reduction}, failure of AMP \cite{LKZ-sparse,LKZ-lowrank}, sum-of-squares lower bounds \cite{MW-sos,sos-detecting}, and the Overlap Gap Property \cite{OGP-submatrix}. This existing work does not yield a precise expression for the optimal runtime (as in Conjecture~\ref{conj:main}).

\paragraph{Comparison to \cite{ding-subexp}.}

The precise runtime achievable in the hard regime was studied by \cite{ding-subexp} in a setting similar to Section~\ref{sec:model} except where $x$ is a sparse Rademacher vector, i.e., the $k$ nonzero entries of $x$ are uniformly $\pm 1$ (instead of our sparse binary setting where the nonzero entries are all 1). In this sparse Rademacher setting, \cite{ding-subexp} conjectured that the runtime $\exp\left(\tilde O\left(\frac{k^2}{\lambda^2 n}\right)\right)$ (achieved by their algorithm) is optimal (up to log factors in the exponent) everywhere in the hard regime $\sqrt{k/n} \ll \lambda \ll \min\{1,k/\sqrt{n}\}$. They gave formal evidence for this conjecture based on the \emph{low-degree likelihood ratio}, a method developed in a recent line of work on the sum-of-squares hierarchy \cite{pcal,HS-bayesian,sos-detecting,hopkins-thesis} (see also \cite{lowdeg-survey} for a survey).

We conjecture (Conjecture~\ref{conj:main}) that the runtime $\exp\left(\tilde O\left(\frac{k^2}{\lambda^2 n}\right)\right)$ is also optimal in the \emph{binary} setting (which could \emph{a priori} admit faster algorithms than the Rademacher setting). While the algorithm of \cite{ding-subexp} still works in the binary setting, the low-degree lower bounds do not. This is because low-degree lower bounds are actually lower bounds against \emph{detection} (which in the Rademacher case is believed to be equally hard as recovery). Since the binary setting admits a trivial detection algorithm (discussed above), low-degree lower bounds are not able to capture the recovery threshold in this setting. In this paper we follow a different point of view and give evidence for Conjecture~\ref{conj:main} based on free energy wells. Thus our main contributions as compared to \cite{ding-subexp} are (i) to corroborate the low-degree lower bounds of \cite{ding-subexp} using a completely different method, and (ii) to give lower bounds in the binary setting where the lower bounds of \cite{ding-subexp} do not apply.

\subsection{Free Energy Wells and Overlap Gap Property}

Over the last decade, an inspiring connection has been drawn in the study of the computational hardness of random optimization problems, between the geometry of the solution space and the algorithmic difficulty of the problem of interest. The connection originated in the study of spin glass systems (see e.g., \cite{TalagrandBook}) and was later used in the study of random satisfiability problems such as random $k$-SAT \cite{mezard2005clustering,AC-barriers,AchlioptasCojaOghlanRicciTersenghi} as well as average-case combinatorial optimization problems such as maximum independent set in random graphs \cite{gamarnik2014limits,rahman2014local}. Specifically, for many models it has been observed that the appearance of certain ``bottleneck" or disconnectivity properties in the solution space such as \textit{Free Energy Wells} and the \textit{Overlap Gap Property (OGP)}, which both originated in spin glass theory, indicate an algorithmic impediment for various classes of algorithms. For instance, in certain settings, variants of OGP have been shown to imply failure of local algorithms~\cite{gamarnik2014limits,rahman2014local,local-maxcut}, WALKSAT~\cite{walksat}, approximate message passing (AMP)~\cite{ogp-amp}, Langevin dynamics~\cite{ogp-lowdeg}, and low-degree polynomials~\cite{ogp-lowdeg}. In various cases, the appearance of such structural properties has been proven to coincide with the conjectured algorithmic hard phase for the problem. Moreover, it has often been observed that in the absence of such properties, even simple local improvement algorithms such as gradient descent can succeed.

Recently, a similar connection has been drawn in the context of statistical tasks (with a ``planted'' signal), between the geometry of the parameter space and the inference task of interest (see e.g., \cite{Zad19} and references therein). For example, the OGP phase transition has been studied in the contexts of high-dimensional linear regression \cite{gamarnikzadik,gamarnikzadik2} and planted clique \cite{gamarnik2019landscape}, and free energy wells have been studied in tensor PCA \cite{alg-tensor-pca}. Both OGP and free energy wells have also been recently studied in the context of sparse PCA but in the regime where the sparsity scales linearly with the dimension \cite{OGP-submatrix}.

While we consider this connection of high interest, it has so far been used in the literature solely for predicting thresholds between ``easy'' and ``hard'' phases. In the current paper we seek to go beyond this by using structural properties to address the following quantitative question: \begin{center} \textit{What is the optimal (sub-exponential) running time for solving the problem in the hard regime?}
\end{center}
Furthermore, as it has been rigorously shown in the models mentioned above, the existence of OGP or free energy wells implies the failure of MCMC methods \cite{gamarnikzadik2,gamarnik2019landscape,OGP-submatrix,alg-tensor-pca}. However, most of the current results apply either for ``low enough'' or ``high enough'' choices of the \textit{temperature} parameter. A main contribution of this work is to establish lower bounds against natural MCMC methods for \emph{all} temperature levels, including temperature levels that depend on $n$.

Furthermore, in the recent work \cite{gamarnik2019landscape} where the OGP for the planted clique model has been analyzed, strong evidence has been provided that the OGP phase transition point is significantly different from the true computational threshold. Interestingly, the authors propose as a solution the study of an ``overparametrized" solution space where indeed they provide evidence that the OGP phase transition takes place at the computational threshold.  The overparametrization takes place in terms of the sparsity parameter of the model (that is, the size of the planted clique). In this work, motivated by the value of ``misparametrization" in both  \cite{gamarnik2019landscape} and \cite{ding-subexp}, we study both the OGP and the free energy wells of the sparse PCA model for a wide range of misparametrized sparsity levels $k'$.

\subsection*{Notation}
Throughout the paper we use standard asymptotic notation. Specifically,
for any real-valued sequences $\{a_n\}_{n \in \mathbb{N}}$ and $\{b_n\}_{n \in \mathbb{N}}$, $a_n=\Theta\left(b_n\right)$
if there exists an absolute constant $c>0$ such that $\frac{1}{c}\le |\frac{a_n}{ b_n}| \le c$; $a_n =\Omega\left(b_n\right)$ or $b_n = O\left(a_n \right)$ if there exists  an absolute constant $c>0$ such that $|\frac{a_n}{ b_n}| \ge c$; $a_n =\omega \left(b_n \right)$ or $b_n = o\left(a_n\right)$ if $\lim_{n \to \infty}| \frac{a_n}{ b_n}| =+\infty$. Furthermore, $a_n=\tilde{\Theta}\left(b_n\right)$
if there exist absolute constants $c,d>0$ such that $\frac{1}{c \log ^d n}\le |\frac{a_n}{ b_n}| \le c \log ^d n$; $a_n =\tilde{\Omega}\left(b_n\right)$ or $b_n = \tilde{O}\left(a_n \right)$ if there exist absolute constants $c,d>0$ such that $|\frac{a_n}{ b_n}| \ge c/ \log^d n$. We use $a_n \ll b_n$ to mean $a_n \le b_n / n^{o(1)}$. An event occurs with \emph{high probability} if it has probability $1-o(1)$.

\section{Preliminaries}

We now focus on the sparse PCA problem as defined in Section~\ref{sec:model}.

\subsection{Boosting}\label{sec:boost}

The following algorithmic observation (used in \cite{ding-subexp}), which is an important consideration throughout this paper, provides conditions under which a sufficiently good approximate estimate of $x$ can be boosted to exact recovery. Specifically, suppose we are able to produce a (not-necessarily-sparse) ``guess'' $v \in \mathbb{R}^n$. (Also suppose for now that $v$ is independent from the noise $W$, but we will explain below why this is not restrictive.) It turns out that if $|\langle v,x \rangle|$ is large enough then we can exactly recover $x$ by thresholding the entries of $Yv$. Specifically, notice that $(Yv)_i$ for $i=1,2,\ldots,n$ is distributed as $ \frac{\lambda}{k} \langle v,x \rangle x_i + \mathcal{N}(0,\sigma_i^2)$ with $\sigma_i^2 \le 2\|v\|^2/n$. To achieve exact recovery by thresholding the entries of $Yv$ it is sufficient to have
\[ \frac{|\langle v,x \rangle|}{\|v\|} \ge (4+\varepsilon) \frac{k}{\lambda} \sqrt{\frac{\log n}{n}} \]
for any constant $\varepsilon > 0$. Here we have used the Gaussian tail bound $\Pr\{\mathcal{N}(0,\sigma^2) \ge t\} \le \exp(-t^2/(2\sigma^2))$ and a union bound over the $n$ indices.

The assumption that $v$ be independent from $W$ is not restrictive because given $Y = \frac{\lambda}{k} xx^\top + W$ it is possible to sample $Y_1 = \frac{\lambda}{k\sqrt{2}} xx^\top + W_1$ and $Y_2 = \frac{\lambda}{k\sqrt{2}} xx^\top + W_2$ where $W_1$ and $W_2$ are distributed as $W$ but are independent from each other (see e.g., Algorithm~4 of \cite{ding-subexp}). Thus we can use $Y_1$ to produce the guess $v$ and then use $Y_2$ for the boosting step, and we only suffer a factor of $\sqrt{2}$ in the signal-to-noise ratio (which is negligible for our purposes).

\subsection{Posterior Distribution}

Taking the point of view of Bayesian inference, the posterior distribution of the signal $x$ given the observation $Y$ is
\[ \Pr[x \,|\, Y] \propto \exp\left(-\frac{n}{4} \left\|Y - \frac{\lambda}{k} xx^\top\right\|_F^2 \right) \propto \exp\left(\frac{\lambda n}{2k}\, x^\top Y x\right). \]
In the language of statistical physics, this is a Gibbs distribution over the $k$-sparse vectors \begin{align}\label{def:sk}S_k = \{v \in \{0,1\}^n \,:\, \|v\|_0 = k\}\end{align} given by $\mu_\beta(v) \propto \exp(-\beta H(v))$
where the Hamiltonian is
\begin{equation}\label{eq:H}
H(v) = -v^\top Y v
\end{equation}
and the inverse temperature is
\begin{equation}\label{eq:bayes} \beta = \beta_\text{Bayes} := \frac{\lambda n}{2k}. \end{equation}
More explicitly,
\[ \mu_\beta(v) = \frac{1}{Z_\beta} \exp(-\beta H(v)) \]
where
\[ Z_\beta = \sum_{v \in S_k} \exp(-\beta H(v)) \]
is the partition function. We will sometimes consider the Gibbs measure with the same Hamiltonian but at different temperatures. We will also sometimes consider the same Hamiltonian $H(v) = -v^\top Y v$ but on the space of $k'$-sparse vectors.

\subsection{Free Energy Wells}\label{sec:FEW}

We will study the existence and depth of \emph{free energy wells}, defined as follows.

\begin{definition}\label{def:FEW}
Consider the Gibbs distribution $\mu_\beta(v) \propto \exp(-\beta H(v))$ on the space of $k'$-sparse vectors $S_{k'} = \{v \in \{0,1\}^n \,:\, \|v\|_0 = k'\}$, with Hamiltonian $H(v) = -v^\top Y v$ and some inverse temperature $\beta \ge 0$. For some $\ell > 0$, let $A = \{v \in S_{k'} \,:\, 0 \le \langle v,x \rangle < \ell\}$ and $B = \{v \in S_{k'} \,:\, \ell \le \langle v,x \rangle \le 2\ell\}$ (where, recall, $x$ is the planted signal). We say that the depth of the free energy well at correlation $\ell$ is
\[ D_{\beta,\ell} := \log \mu_\beta(A) - \log \mu_\beta(B). \]
If $D_{\beta,\ell} \le 0$ then there is no free energy well at correlation $\ell$.
\end{definition}
This is similar to the notion of free energy wells used in \cite{alg-tensor-pca,OGP-submatrix,gamarnik2019landscape}. In contrast to some previous definitions, we do not explicitly require that the third region $C = \{v \in S_{k'} \,:\, \langle v,x \rangle > 2\ell\}$ satisfies $\mu_\beta(C) > \mu_\beta(B)$ (although this will typically be true in our setting). Our notion of free energy well is designed to imply that MCMC methods take a long time to exit region $A$ (which does not require a condition on region $C$).

More specifically, if a Markov chain with stationary distribution $\mu_\beta$ is initialized according to the conditional distribution $\mu_\beta(\cdot | A)$, it requires time $\gtrsim \exp(D_{\beta,\ell})$ to escape from region $A$. This follows from standard arguments, which we repeat here for convenience. Consider the undirected graph $\mathcal{G}$ of $\binom{n}{k'}$ vertices, where each vertex corresponds to a unique binary $k'$-sparse vector and we connect two vertices if the Hamming distance between their associated vectors is exactly 2 (which is the minimal nozero distance). Also add a self-loop on every vertex. Let $X_0 \sim \mu_\beta(\cdot|A)$ and let $X_0, X_1, X_2\ldots$ be any Markov chain on the vertices of $\mathcal{G}$ (with transitions allowed only on the edges of $\mathcal{G}$) whose stationary distribution is $\mu_\beta$. One canonical choice for this Markov chain is the Metropolis chain (see e.g., \cite[Section 3.1]{levin2017markov}), which uses the following update step: if the current state is $v$, choose a random neighbor $u$ (not equal to $v$) and move to $u$ with probability $\min\{1,\exp(\beta H(v) - \beta H(u))\}$; otherwise remain at $v$.

Define the hitting time, \[\tau_{\beta}:=\inf\{t \in \mathbb{N} \,:\, X_{t} \not \in A\},\]
and the following proposition holds.
\begin{proposition}\label{prop:mcmc}
Consider a fixed $Y$ for which $\mu_\beta$ has a free energy well of depth $D_{\beta,\ell}$ at correlation $\ell$. As above, let $X_0,X_1,\ldots$ be any Markov chain on $\mathcal{G}$ with stationary distribution $\mu_\beta$, initialized from $X_0 \sim \mu_\beta(\cdot|A)$. Then for any $t \ge 1$,
\[ \Pr\{ \tau_\beta \le t \} \le t \exp(-D_{\beta,\ell}). \]
\end{proposition}
The (very simple) proof can be found in Appendix \ref{app:mcmc}. This shows that certain MCMC methods with the initialization $\mu(\cdot|A)$ require time $\gtrsim \exp(D_{\beta,\ell})$. We expect that the same should hold true for uniformly random initialization, although we do not have a proof. A more general class of MCMC methods than the one we consider would allow transitions that can increase or decrease the sparsity $k'$; while we do not expect this should help, our results do not apply to such methods. Note that Proposition~\ref{prop:mcmc} applies to Markov chains that are not necessarily reversible; however, we do require that the stationary distribution be $\mu_\beta$, which does not hold for most natural non-reversible dynamics. Finally, note that free energy wells do not necessarily imply failure of various other common algorithmic approaches such as power iteration, message passing, or semidefinite programming.

\paragraph{Informative values for $\ell$ and $k'$.}

A free energy well is only meaningful for certain values of $\ell$ and $k'$. Certainly we need $1 \le 2\ell \le \min\{k,k'\}$ because $B$ should be nonempty and $\min\{k,k'\}$ is the largest possible value for $\langle v,x \rangle$. Note that a random guess $v \in S_{k'}$ achieves correlation $\langle v,x \rangle \approx \frac{kk'}{n}$, so a meaningful free energy well should also have $\ell \gg \frac{kk'}{n}$ so that a uniformly random initialization falls in region $A$ with high probability. Finally, we would like to show the existence of a free energy well at some correlation $\ell \le \frac{k}{2\lambda}\sqrt{\frac{k'}{n} \log n}$ because if an algorithm could reach a $v \in S_{k'}$ with $\langle v,x \rangle$ larger than this (by a constant), this can be boosted to exact recovery (see Section~\ref{sec:boost}). Hence, our goal will be to show there is a free energy well of large depth at some correlation
\begin{equation}\label{eq:ell-good}
\max\left\{1,\frac{kk'}{n}\right\} \ll \ell \le \frac{k}{2\lambda}\sqrt{\frac{k'}{n} \log n}.
\end{equation}
The requirement $2\ell \le \min\{k,k'\}$ from above will be subsumed by~\eqref{eq:ell-good} due to the conditions~\eqref{eq:k-good} on $k'$ discussed below.

We are only interested in $k'$ values for which solving the misparametrized optimization problem~\eqref{eq:argmax} is actually useful for recovering the signal $x$. More precisely, we require that if $v$ has maximal correlation $\langle v,x \rangle = \min\{k,k'\}$ then it can be boosted to exact recovery, i.e., we assume
\begin{equation}\label{eq:k-good}
\min\{k,k'\} \ge \frac{k}{\lambda}\sqrt{\frac{k'}{n} \log n}.
\end{equation}
Note that this is equivalent to the following bounds on $k'$:
\[ \frac{k^2 \log n}{\lambda^2 n} \le k' \le \frac{n \lambda^2}{\log n}. \]
Throughout, we will refer to ``informative'' $\ell$ values as those satisfying~\eqref{eq:ell-good}, and to ``informative'' $k'$ values as those satisfying~\eqref{eq:k-good}.

We note that for concreteness, we have fixed a particular choice for the constant in front of the boosting threshold $\frac{k}{\lambda} \sqrt{\frac{k'}{n} \log n}$. This is a conservative lower bound on the threshold at which the boosting procedure succeeds. However, our results are not sensitive to the specific choice of this constant.

\subsection{Overlap Gap Property}

We study the Overlap Gap Property (OGP), which is formally defined in terms of the near-optimal solutions of a naturally associated optimization problem. As explained in the introduction, it has been repeatedly observed that the appearance of OGP in the solution space indicates an algorithmic barrier (at least for some classes of algorithms) for solving the optimization problem of interest, sometimes matching the conjectured computational hardness threshold. Naturally, the optimization problems for which we study the OGP are \begin{align}\label{phik'}
\Phi_{k'} \; : \; \min  H(v) \;\text{ s.t.\ }\; v \in S_{k'}
\end{align}
where $H$ is the Hamiltonian~\eqref{eq:H}, $S_{k'}$ is the set of $k'$-sparse vectors~\eqref{def:sk}, and $k'$ is a possibly-misparametrized sparsity level. OGP is motivated by the study of concentration of the associated Gibbs measures (see \cite{TalagrandBook}) for low enough temperature, and therefore concerns the geometry of the near-optimal solutions. Informally, the variant of OGP typically used for statistical inference problems  (see e.g., \cite{gamarnik2019landscape}) states that any near-optimal solution of the optimization problem has a correlation with the true signal that is either very large or very small.

Formally, for the optimization problem $\Phi_{k'}$, we define $k'$-OGP as follows.

\begin{definition}[$k'$-OGP]\label{def:OGP} $\Phi_{k'}$ exhibits the $k'$-Overlap Gap Property ($k'$-OGP) if for some $\zeta_{1,n},\zeta_{2,n} \in \{0,1,\ldots,\min\{k',k\}\}$ with $\zeta_{2,n} > \zeta_{1,n} + 2$ and some $r_n \in \mathbb{R}$, the following properties hold.
\begin{itemize}
\item[(1)] There exist $v,w \in S_{k'}$ with $\inner{v,x} \leq \zeta_{1,n},$ $\inner{w,x} \geq \zeta_{2,n}$ and $\max \{ H(v), H(w)\} \leq r_n$.
\item[(2)] For any $v \in S_{k'}$ with $H(v) \leq r_n$, it holds that either $\inner{v,x} \leq \zeta_{1,n}$ or $\inner{v,x} \geq \zeta_{2,n}.$
\end{itemize} 
\end{definition}

Condition (2) states that there is a forbidden region for overlap values $\langle v,x \rangle$ of near-optimal vectors $v$. Condition (1) ensures that overlaps on both sides of the gap can actually be realized.

We can see that OGP implies failure of a certain class of ``local'' algorithms (with worst-case initialization) as follows. Consider any algorithm that keeps track of a vector $v \in S_{k'}$ and iteratively updates it to a neighboring vector in $S_{k'}$ (i.e., a vector differing in 2 coordinates) so that the objective $H(v)$ always improves (i.e.\ decreases). If the current state $v$ satisfies $\langle v,x \rangle \le \zeta_{1,n}$ and $H(v) \le r_n$ then the algorithm will never be able to reach $x$ because it is ``stuck'' on the wrong side of the gap.

More generally, if $k'$-OGP holds with $\zeta_{2,n} - \zeta_{1,n} > \Delta$, this implies failure of a broader class of local algorithms that can change $\Delta$ coordinates at each step. Our results will in fact establish the presence of a large ``gap'' $\zeta_{2,n} - \zeta_{1,n} = \omega(\sqrt{k'})$ (see Theorem~\ref{thm:OGPLow}).

\section{Main Results}

\subsection{Overview}

Recall that our main result is a lower bound on the depth of free energy wells that gives evidence in favor of Conjecture~\ref{conj:main}.

\begin{theorem}\emph{(Main result, informal)}\label{thm:main-informal}
Suppose $\lambda$ is in the ``hard'' regime $\sqrt{k/n} \ll \lambda \ll \min\{1,k/\sqrt{n}\}$, $k'$ is ``informative'' (satisfying~\eqref{eq:k-good}), and that additionally, either (i) $k' \le k$ or (ii) $\lambda \ll (k/n)^{1/4}$. For any $\beta \ge 0$ (possibly depending on $n$), there exists an ``informative'' $\ell$ (satisfying~\eqref{eq:ell-good}) such that $D_{\beta,\ell} \ge \tilde\Omega\left(\frac{k^2}{\lambda^2 n}\right)$ with high probability.
\end{theorem}

\begin{remark}
If $k \ll n^{1/3}$ then condition (ii) in Theorem~\ref{thm:main-informal} is automatically satisfied due to the bounds on $\lambda$ (namely the assumption $\lambda \ll k/\sqrt{n}$).
\end{remark}

\subsection{Proof Techniques}\label{sec:pf-tech}

Our main lower bound on the depth of free energy wells is deduced by combining two different lower bounds: one for high temperature ($\beta$ small) and one for low temperature ($\beta$ large). In particular, we have the following main results:

\begin{itemize}
    \item[(1)] {\bf High-temperature lower bound} (see Corollary~\ref{cor:hightemp-final}): If $\lambda \ll \min\{1,k/\sqrt{n}\}$, $\beta \ll \min\{\frac{k}{\lambda},\frac{n}{\lambda k'}\}$, and $k'$ is informative, then there is an informative $\ell$ for which $D_{\beta,\ell} = \Omega\left(\min\left\{\frac{k}{\beta \lambda},\frac{k}{\lambda}\sqrt{\frac{k'}{n}}\right\}\right)$ with high probability.
    
    \item[(2)] {\bf Low-temperature lower bound} (see Theorem~\ref{thm:FEWLow} and Appendix~\ref{app:increase-lambda}): Suppose that $\sqrt{k/n} \ll \lambda \ll \min\{1,k/\sqrt{n}\}$, $k'$ is informative, and $k' \ll \sqrt[3]{\frac{k^2 n}{\lambda^2}}$. If $\beta \gg \beta_{\mathrm{Bayes}} = \frac{\lambda n}{2k}$ then there is an informative $\ell$ for which $D_{\beta,\ell} = \Omega(k' \log(n/k'))$ with high probability.
    
    \item[(3)] {\bf Overlap Gap Property} (see Theorem \ref{thm:OGPLow}): If $\sqrt{k/n} \ll \lambda \ll \min\{1,k/\sqrt{n}\}$, $k'$ is informative\footnote{Technically, our result on OGP (Theorem~\ref{thm:OGPLow}) requires an assumption on $k'$ that is slightly stronger (by multiplicative logarithmic factors) than ``informative''.}, and $k' \ll \sqrt[3]{\frac{k^2 n}{\lambda^2}}$, then $k'$-OGP holds with high probability. In particular, if $k \ll n^{1/3}$ or $k' \le k$ then $k'$-OGP holds for the entire ``possible but hard" regime.
    
\end{itemize}

\noindent Our main result (Theorem~\ref{thm:main-informal}) follows from combining (1) and (2) in a straightforward way; the details are deferred to Appendix~\ref{app:main-pf}.

To prove the high-temperature lower bound, we follow an argument that was used by \cite{alg-tensor-pca} to show the existence of free energy wells in tensor PCA. The idea is to leverage ``entropy'', i.e., the fact that there are many more vectors $v \in S_{k'}$ with small correlation $\langle v,x \rangle$ than with large correlation, and this effect overpowers the strength of the signal (if temperature is sufficiently high).

The core technical contribution of this paper is to establish the presence of OGP. This requires both a first moment argument and a second moment argument. The key arguments towards proving it can be found in Section \ref{sec:lowLemmmas}.

The low-temperature lower bound is deduced using the same tools that prove OGP. The presence of OGP can be thought of as a free energy well at zero temperature ($\beta = \infty$). Furthermore, for sufficiently low temperature levels the existence of sufficiently deep free energy wells can be established as a direct corollary of OGP. A slightly more involved use of the tools developed to prove OGP allows us to prove the exact low-temperature lower bound described above and in Theorem~\ref{thm:FEWLow}.

\subsection{Low Temperature}

In this section we formally state our results in the low-temperature regime, with proofs deferred to Sections \ref{sec:FEWLow} and \ref{sec:OGP} (and with machinery leading up to these proofs developed in Sections \ref{sec:lowLemmmas} through \ref{sec:2MM}). For the results in this section, care has been taken to give as strong results as possible, often at the level of logarithmic factors. We start with the parameter assumptions under which the results in this section hold. We assume that the signal-to-noise ratio $\lambda$ satisfies the following.
\begin{assumption}\label{assum:LambdaMain}
 We assume \begin{align}\label{assum:lambda}
 \lambda=\omega \left( \frac{k}{n }\log \left( \frac{n}{k} \right) \right) \quad \text{and } \quad \lambda=o\left(\min\left\{1,\frac{k}{\sqrt{n}\log n}\right\}\right).
 \end{align}
 \end{assumption}
Up to logarithmic factors, this regime for $\lambda$ covers the entire ``hard" regime $\sqrt{k/n} \ll \lambda \ll \min\{1,k/\sqrt{n}\}$. We have only required the weaker lower bound $\lambda \gg k/n$ here, but $\lambda \gg \sqrt{k/n}$ will be implied by~\eqref{k'Assum1} below. Next we restrict the values of $k'$.
\begin{assumption}\label{assum:k'Main} We assume 
\begin{align}\label{k'Assum1}
    k'=\omega\left( \frac{k^2}{\lambda^2 n} \,\frac{ \log^2 \left( \frac{n}{k}\right)}{\log \left( \frac{\lambda n}{k \log(n/k)} \right)}\right),\quad k'=o\left(\frac{\lambda^2 n}{\log^3 n}\right), \end{align}
    and
    \begin{align}\label{addit_assum}k'=o\left( \left(\frac{k^2 n}{\lambda^2 \log n}\right)^{\frac{1}{3}} \right).
\end{align}
\end{assumption}
The regime for $k'$ covered in \eqref{k'Assum1} matches, up to logarithmic factors, the informative values of $k'$. While this does not quite cover the full range of informative $k'$ values (as defined in~\eqref{eq:k-good}) due to the log factors, a simple argument allows us to extend our main theorem on free energy wells (Theorem~\ref{thm:FEWLow}) to the remaining $k'$ values by increasing $\lambda$ slightly; see Appendix~\ref{app:increase-lambda} for details. The effect of the additional assumption \eqref{addit_assum} is discussed in the following remark.

\begin{remark} \label{rem:addit}
The additional assumption \eqref{addit_assum} is believed to be of a technical nature. In the case where $\lambda=o\left( \sqrt[4]{\frac{k}{n}} \log n \right)$, the second assumption in \eqref{k'Assum1} implies \eqref{addit_assum} by direct comparison. If instead $k' \le k$ then~\eqref{addit_assum} holds provided $\lambda = o\left(\sqrt{\frac{n}{k \log n}}\right)$, which is, up to log factors, implied by the ``hard'' regime for $\lambda$ (namely $\lambda \ll 1$). Thus the technical assumption~\eqref{addit_assum} is not needed when either $\lambda \ll \sqrt[4]{k/n}$ or $k' \le k$.
\end{remark}

Our first result is on the existence of sufficiently deep free energy wells at low temperature. Recall the definition $\beta_{\mathrm{Bayes}} := \frac{\lambda n}{2k}$ from equation \eqref{eq:Bayes}.

\begin{theorem}\label{thm:FEWLow} Suppose $k,k'=o(n)$, $k,k'=\omega\left(\sqrt{\log n}\right)$, and $\ell=\Theta\left(\frac{k}{\lambda}\sqrt{\frac{k'}{n}\log \frac{n}{k'}}\right)$.
Under Assumptions \ref{assum:LambdaMain} and \ref{assum:k'Main}, the following holds for some universal constants $d_1,c_0>0$ and $d_2>1$. If $\ell \leq c_0 \frac{k}{\lambda}\sqrt{\frac{k'}{n}\log \frac{n}{k'}}$ then 
\begin{align} \label{depthOGP}
   d_1 \left(\frac{\beta}{\beta_{\mathrm{Bayes}}}-d_2\right)k'\log \left( \frac{n}{k'}\right)\leq D_{\beta,\ell} \leq  d_1\left(\frac{\beta}{\beta_{\mathrm{Bayes}}}+d_2\right)k'\log \left( \frac{n}{k'}\right)
    \end{align}
    with high probability as $n \to +\infty$. In particular, for any $\beta=\omega\left(\beta_{\mathrm{Bayes}}\right)$ it holds
    \begin{align}\label{eq:Bayes}
    D_{\beta,\ell}=\Theta\left(\frac{\beta}{\beta_{\mathrm{Bayes}}}k'\log \left( \frac{n}{k'}\right) \right)=\Theta\left(\beta\frac{kk'}{\lambda n}\log \left( \frac{n}{k'}\right) \right)=\omega\left(k'\log \left( \frac{n}{k'}\right) \right)
    \end{align}
    with high probability as $n \to +\infty$.
\end{theorem}The proof of Theorem \ref{thm:FEWLow} can be found in Section \ref{sec:FEWLow}.

Notice that the choice of overlap $\ell=\Theta\left(\frac{k}{\lambda}\sqrt{\frac{k'}{n}\log \frac{n}{k'}}\right)$ falls exactly at the edge of informative values of overlap as discussed in Section \ref{sec:FEW}. Hence, Theorem \ref{thm:FEWLow} provides, under our assumptions, for all $\beta=\omega\left(\beta_{\mathrm{Bayes}}\right)$, the exact (up to constants) depth of the free energy wells at the edge of the informative overlaps $\ell=\Theta\left(\frac{k}{\lambda}\sqrt{\frac{k'}{n}\log \frac{n}{k'}}\right)$. As discussed above, we explain in Appendix~\ref{app:increase-lambda} how to extend the assumption~\eqref{k'Assum1} to cover all informative $k'$ values. Thus, based on Remark \ref{rem:addit}, in the case $\lambda \ll \sqrt[4]{k/n}$, the result applies for the whole range of informative $k'$ values.

Our next result is on the Overlap Gap Property as defined in Definition \ref{def:OGP}. We prove that under Assumptions \ref{assum:LambdaMain} and \ref{assum:k'Main}, $k'$-OGP indeed holds, providing evidence for the computational hardness of the model under these assumptions.

\begin{theorem}\label{thm:OGPLow}
Suppose $k',k=o\left(n\right)$ and $k,k'=\omega\left(\sqrt{ \log n} \right)$. Under Assumptions \ref{assum:LambdaMain} and \ref{assum:k'Main}, the optimization problem $\Phi_{k'}$ exhibits the $k'$-OGP with $\zeta_{2,n}-\zeta_{1,n}=\omega\left(\sqrt{k'} \right),$ with high probability as $n \rightarrow +\infty.$
\end{theorem}
The proof of Theorem \ref{thm:OGPLow} can be found in Section~\ref{sec:OGP}. Similarly to above, based on Remark \ref{rem:addit}, in the case $\lambda=o\left( \sqrt[4]{k/n} \log n\right)$, the result applies for essentially all informative $k'$ values.

\subsection{High Temperature}

In this section we present our lower bound on depth in the high-temperature regime, with proofs deferred to Section~\ref{sec:pf-high-temp}. These results are sometimes loose by factors of $n^{o(1)}$ in favor of simpler statements and proofs. We give non-asymptotic results which hold for all specified values of the parameters, not only in the large-$n$ limit. We first present a general result that holds for any $\ell$, and later specialize to the best choice of $\ell$ (Corollary~\ref{cor:hightemp-final}). The proof of the following core theorem can be found in Section~\ref{sec:pf-high-temp}.

\begin{theorem}\label{thm:ent-depth}
Fix a constant $\delta > 0$. For any $n$ exceeding some $n_0 = n_0(\delta)$, for any $k$, for any $k' \le n^{1-\delta}$, for any $k$-sparse signal $x \in \{0,1\}^n$, for any $\beta \ge 0$, and for any $\ell \ge 2e k (k'/n)^{1-\delta}$ (with $1 \le 2\ell \le \min\{k,k'\}$), with probability at least $1-2^{-(\ell-2)/2}$ over $W$, the free energy well at correlation $\ell$ has depth bounded by
\begin{equation}\label{eq:D-result}
D_{\beta,\ell} \ge -\frac{4 \beta \lambda}{k} \ell^2 + \frac{\log 2}{2}\, \ell - \log 2.
\end{equation}
\end{theorem}
The requirement $\ell \ge 2ek(k'/n)^{1-\delta}$ is not restrictive because an informative $\ell$ should satisfy $\ell \gg kk'/n$ so that a uniformly random vector $v \in S_{k'}$ lands in the set $A$ (from Definition~\ref{def:FEW}) with high probability (see~\eqref{eq:ell-good}). More specifically, $\ell \ge 2ek(k'/n)^{1-\delta}$ implies $|A|/{n \choose k'} \ge 1 - 2^{1-\ell}$ (see~\eqref{eq:A-small}).

We now specialize to our specific range of informative $\ell$ values~\eqref{eq:ell-good} and informative $k'$ values~\eqref{eq:k-good} to obtain our main result. The proof can be found in Section~\ref{sec:pf-high-temp}.

\begin{corollary}\label{cor:hightemp-final}
Fix a constant $\delta > 0$ and suppose $n \ge n_0(\delta)$ and $k' \le n^{1-\delta}$. Also suppose $\min\{k,k'\} \ge \frac{k}{\lambda}\sqrt{\frac{k'}{n}}$ and $\lambda \le \min\left\{n^{-\delta} \frac{k}{2} \sqrt{\frac{k'}{n}}, \frac{1}{4e} \left(\frac{n}{k'}\right)^{1/2-\delta}\right\}$. For any $\beta$ satisfying
\[ 0 \le \beta \le \frac{\log 2}{16} n^{-\delta} \min\left\{\frac{k}{\lambda},\frac{n}{2e\lambda k'}\right\} \]
there exists $\ell$ satisfying
\[ \max\left\{n^\delta,2ek(k'/n)^{1-\delta}\right\} \le \ell \le \frac{k}{2\lambda}\sqrt{\frac{k'}{n}} \le \frac{1}{2} \min\{k,k'\} \]
such that with probability at least $1-2^{-(\ell-2)/2} \ge 1-2^{-(n^\delta-2)/2}$,
\[ D_{\beta,\ell} \ge \frac{\log 2}{8} \min\left\{\frac{\log 2}{8} \frac{k}{\beta \lambda}, \frac{k}{\lambda} \sqrt{\frac{k'}{n}} \right\} - \log 2. \]
\end{corollary}
Note that since $1 \le k' \le n$, the condition on $\lambda$ is implied by our regime for $\lambda$, namely $\lambda \ll \min\{1,k/\sqrt{n}\}$.

\subsection{Discussion and Future Directions}

We highlight a few conceptual contributions of the present work to the study of the solution space of inference problems. 
\begin{itemize}
\item First, by combining two different lower bounds we showed the existence of free energy wells at \emph{all} values of the inverse temperature $\beta$, possibly depending on $n$. In contrast, prior work (e.g., \cite{gamarnikzadik2,gamarnik2019landscape,OGP-submatrix,alg-tensor-pca}) has often been restricted to only some regime of temperature values, sometimes missing the \emph{Bayesian temperature} (which is a particularly natural choice of temperature, as the Gibbs measure corresponds to the posterior distribution).
\item Second, we investigated not just the threshold between ``easy'' (polynomial time) and ``hard'', but the precise subexponential runtime throughout the hard regime. This required precise bounds on the depth of free energy wells instead of simply the existence of free energy wells. Interestingly, these bounds coincided with the conjectured runtime obtained from the low-degree likelihood ratio. The extent to which the connection between the depth of free energy wells and optimal runtime extends beyond the sparse PCA model, remains an interesting question for future work.
\end{itemize}

Some directions for future work are as follows. One question is whether the condition $\lambda \ll (k/n)^{1/4}$ can be removed from our main result. Another unexplored direction is whether it is possible to give positive results for MCMC algorithms. While we have shown (in a large parameter regime) that MCMC cannot improve upon the existing algorithms of \cite{ding-subexp,anytime-pca}, it remains unclear (aside from the trivial case in Remark~\ref{rem:tight}) whether MCMC methods can match this performance (and if so, for which parameters $k', \beta$). It seems that we are currently lacking tools to prove algorithmic results for MCMC methods applied to high-dimensional inference problems. For example, the influential work of \cite{jerrum} showed that MCMC methods fail to find large cliques in random graphs, giving perhaps the first concrete evidence for the famous statistical-to-computational gap in the planted clique problem; however, we are still lacking a matching upper bound showing that MCMC can indeed find the clique in the ``easy'' regime where other polynomial-time methods are known. One success in showing positive results for MCMC (or rather for the continuous-time analogue, \emph{Langevin dynamics}), is the ``bounding flows'' method \cite{bounding-flows} used by \cite{alg-tensor-pca} for the tensor PCA problem. Note that in this case, MCMC methods (and conjecturally, all ``local'' algorithms) have strictly weaker performance than the best known polynomial-time algorithms (see \cite{alg-tensor-pca}).

\section{Auxiliary Low Temperature Results}\label{sec:lowLemmmas}

This section states several results on the structure of the solution space of $\Phi_{k'}$ as defined in \eqref{phik'}. These results are later used to establish Theorems \ref{thm:FEWLow} and \ref{thm:OGPLow}.

We denote by $\phi_{k'}$ the optimal value of $\Phi_{k'}$. We define the following restricted version of $\Phi_{k'}$ based on the overlap between the feasible vectors and the ground truth vector $x$. For $\ell=\floor*{kk'/n},\floor*{kk'/n}+1,\ldots,n$ we consider the optimization problems \[\Phi_{k'}(\ell): \min H(v)  \text{ s.t. } v \in \{0,1\}^n, \|v\|_0=k', \langle v,x \rangle=\ell \] and denote by $\phi_{k'}(\ell)$ the optimal value of $\Phi_{k'}(\ell)$.  Notice that we consider only the values of overlap starting from $\floor{kk'/n}$, since that level of overlap can be achieved by a uniformly at random chosen $k'$-sparse binary vector.

Our main technical tool is a tight deterministic approximation of the function $\phi_{k'}(\ell)$ for a wide range of the values of $\ell$ by a deterministic quantity $\Gamma_{k'}\left(\ell\right)$. The result follows from a careful application of the first and second moment method. The quantity is defined as follows.
\begin{definition}
We define the curve $\Gamma_{k'}: \{\floor*{kk'/n},\floor*{kk'/n}+1,\ldots,\min\{k',k\}\} \rightarrow \mathbb{R},$ given by \begin{align}  \Gamma_{k'}\left(\ell\right):=-\lambda \frac{\ell^2}{k}- 2k'\sqrt{\frac{1}{n} \log \left[ \binom{k}{\ell} \binom{n-k}{k'-\ell} \right]}. \end{align} 
\end{definition} 

The approximation result is as follows.
\begin{theorem}\label{FM} Let $n,k,k'\in \mathbb{N}$ with $\max\{k,k'\} \leq n$ and $n \rightarrow +\infty.$ 
\begin{itemize}
    \item[(1)]  Suppose $\alpha_n>0,n \in \mathbb{N}$ be a sequence with $\lim_{n \rightarrow + \infty} \frac{\alpha_n}{\log \left( \min\{k',k\} \right)}=+\infty$. The following is true with high probability as $n \rightarrow +\infty$.
If $\ell \in \mathbb{N}$ with $\floor*{kk'/n} \leq \ell \leq \min\{k',k\}$, then it holds
\begin{align} 
\phi_{k'}(\ell) \geq \Gamma_{k'}\left(\ell\right)- \sqrt{\frac{(k')^2}{n \log \left( \binom{k}{\ell} \binom{n-k}{k'-\ell} \right)}}\alpha_n.
\end{align}

    \item[(2)] Suppose $k'=o\left(n\right),$ $k=o\left(n\right)$, and $k,k'=\omega\left( \sqrt{\log n}\right).$ There exist universal constants $c,C>0$ such that the following holds with high probability as $n \rightarrow +\infty$.  If $\ell \in \mathbb{N}$ with $\floor{kk'/n} \leq \ell \leq c \min \{k',k\}$ then

\begin{align}\label{phiK}
    \phi_{k'}(\ell) \leq \Gamma_{k'}\left(\ell\right)+C\sqrt{\frac{k'\log n}{n} \max\left\{\frac{ \ell^4}{k^2}, \frac{(k')^4}{n^2},k' \right\}}.\end{align}
\end{itemize}
\end{theorem}The proof of part (1) of Theorem \ref{FM} can be found in Section \ref{sec:1MM} and the  proof of part (2) can be found in Section \ref{sec:2MM}.

\begin{remark}
Due to the wide range of the parameters $k,k',n,\lambda,\ell$ considered in the statement of Theorem \ref{FM}, no simplification is globally possible in the maximum of the three terms in \eqref{phiK}.
\end{remark}

Of distinct importance to us is the following corollary of Theorem \ref{FM}.

\begin{corollary}\label{cor:approx}   Let $n,k,k'\in \mathbb{N}$ with $\max\{k,k'\} \leq n$ and $n \rightarrow +\infty.$   Suppose $k'=o\left(n\right),$ $k=o\left(n\right)$, and $k,k'=\omega\left( \log n\right)$ and that Assumption \ref{assum:LambdaMain} and Assumption \ref{assum:k'Main} hold.
Let $\ell$ be such that  $\ell=O\left( \frac{k}{\lambda}\sqrt{ \frac{k'}{n} \log \frac{n}{k'}}\right).$ Then $$\max_{m=\floor{k'k/n},\floor{k'k/n}+1,\ldots,\ell} |\phi_{k'}\left(m \right)-\Gamma_{k'}\left(m\right)|=o\left(\frac{k'k}{\lambda n}\right),$$with high probability as $n \rightarrow + \infty.$
\end{corollary}The proof of the corollary can be found in Appendix~\ref{app:approx}.

We study the monotonicity of the curve $\Gamma_{k'}.$ The following theorem holds.

\begin{theorem}\label{thm:FM2} Let  $\lambda>0$ and $k,k',n \in \mathbb{N}$ with $k',k \leq n$ and $k,k',n \rightarrow + \infty$ and $k',k=o\left(n\right).$ Under Assumption \ref{assum:LambdaMain} and Assumption \ref{assum:k'Main} the following hold.
  For $$\ell_c:=\frac{1}{2\lambda}k\sqrt{\frac{k'}{n  \log \left(\frac{n}{k'}\right)}} \log \left( \frac{1}{2\lambda}\sqrt{\frac{n}{k'  \log \left(\frac{n}{k'}\right)}}  \right),$$it holds $\ell_c=O\left(\frac{1}{\lambda}k\sqrt{\frac{k'}{n}\log \left(n\right)} \right)$ and moreover

\begin{align}\label{ineqov}
    \ell_c= \omega\left(\frac{k'k}{n}\right) \text{ and } \ell_c =o\left( \min\{k',k\}\right).
\end{align}
Furthermore, for arbitrary fixed $\delta \in (0,1)$ and arbitrarily small $\epsilon>0$ the following properties hold for sufficiently large values of $n$.
\begin{itemize}

    \item[(a)] If  $\floor{kk'/n} \leq \ell \leq \floor{\left(1-\delta \right)\ell_c}-1$, then
\begin{align}
\Gamma_{k'}(\ell+1)-\Gamma_{k'}(\ell)& \geq  \delta\left(1-\epsilon\right)  \sqrt{\frac{ k'}{n  \log \left(\frac{n}{k'}\right)}} \log \left(\frac{\left(\ell+1\right) n}{k' k }\right).
    \end{align} In particular, $\Gamma_{k'}$ is strictly increasing for these values of $\ell$.
    \item[(b)]  If  $\floor{kk'/n} \leq \ell$ with $\ell=o \left( \min\{k',k\}\right)$, then for large enough $n$,
\begin{align}
\Gamma_{k'}(\ell+1)-\Gamma_{k'}(\ell)& \leq  \left(1+\epsilon\right)  \sqrt{\frac{ k'}{n  \log \left(\frac{n}{k'}\right)}} \log \left(\frac{\left(\ell+1\right) n}{k' k }\right) \leq   \left(1+\epsilon\right) \sqrt{\frac{ k'}{n  }\log \left(\frac{n}{k'}\right)}.
    \end{align} 
    
    \item[(c)] For any $\ell \geq  10\ceil{\ell_c}-1 ,$
    
    \begin{align}
\Gamma_{k'}(\ell+1)-\Gamma_{k'}(\ell)& \leq  -\lambda \frac{\ell}{k}.
\end{align}
    
\item[(d)] Suppose for some $0<\delta_0<\delta_1<1$ it holds $\delta \in (\delta_0,\delta_1)$. Then for some constant $C>10$ and $D=D(\delta_0,\delta_1)>0$ depending only on $\delta_0$ and $\delta_1$, for any $\delta \in (\delta_0,\delta_1)$ it holds

$$ \Gamma_{k'}\left(\floor{\left(1-\delta \right)\ell_c}\right) - \mathrm{Gap}_n \geq \max\{\Gamma_{k'}\left(\floor{k'k/n} \right),\Gamma_{k'}\left(C\ceil{\ell_c}\right)\} $$for
\begin{align} \label{gap}
    \mathrm{Gap}_n:= D \frac{k'k}{\lambda n}.
    \end{align}

\end{itemize}
\end{theorem}The proof of Theorem \ref{thm:FM2} can be found in Section \ref{Sec:AnalysisFM}.

The connection between the OGP and the optimization of $\Phi_{k'}$, and in particular the values $\phi_{k'}(\ell),\ell=\floor{k'k/n},\ldots,n,$ is established using the following simple proposition.

\begin{proposition}\label{prop:OGP}
Suppose that for some overlap sizes $0 \leq \ell_1 \leq z_{1}<z_{2}-1<z_2 \leq \ell_2 \leq k$ it holds $$\max\{ \phi_{k'}\left(\ell_1\right),\phi_{k'}\left(\ell_2\right)\} <\min_{\ell \in (z_{1},z_{2})} \phi_{k'}\left(\ell\right).$$with high probability as $n \rightarrow +\infty$. Then $\Phi_{k'}$ exhibits the $k'$-OGP with choice of $\zeta_{1,n}=z_1$ and $\zeta_{2,n}=z_2.$ 
\end{proposition} 

The proof of Proposition \ref{prop:OGP} can be found in Appendix~\ref{app:OGP}.

The connection between free energy wells and the optimization of $\Phi_{k'}$ is established using the following simple proposition.

\begin{proposition}\label{prop:FEW}
Let $\lambda>0, \beta>0$ and $k,k',n, \ell \in \mathbb{N}$ with $k',k \leq n$ and $\floor{\frac{k'k}{n}} \leq \ell \leq \frac{k}{2}$. It holds
\begin{align}
\bigg{|}D_{\beta,\ell}-\beta \left[\min_{m=\floor{\frac{k'k}{n}},\floor{\frac{k'k}{n}}+1,\ldots,\ell} \phi_{k'}(m)-\min_{m=\ell,\ell+1,\ldots,2\ell} \phi_{k'}(m) \right] \bigg{|}\leq \log \binom{n}{k'}.
\end{align}
\end{proposition}The proof of Proposition \ref{prop:FEW} can be found in Appendix \ref{app:FEW}.

\section{Proof of Theorem \ref{thm:FM2} - Analysis of the First Moment Curve}\label{Sec:AnalysisFM}
\begin{proof}[Proof of Theorem \ref{thm:FM2}]

We start by proving the first part of the Theorem.
First, notice that since by Assumption \ref{assum:LambdaMain}, $\lambda=\omega\left( \frac{k}{n} \right)$, it holds $$\ell_c =O\left( \frac{1}{\lambda}k\sqrt{\frac{k'}{n  \log \left(\frac{n}{k'}\right)}} \log \left( \frac{n}{k}\sqrt{\frac{n}{k'  \log \left(\frac{n}{k'}\right)}}  \right)\right)=O\left( \frac{1}{\lambda}k\sqrt{\frac{k'}{n\log n}} \right),$$as we wanted.
We continue with establishing \eqref{ineqov}.
We have from Assumption \ref{k'Assum1}
$$k'=\omega\left( \frac{k^2}{\lambda^2 n} \frac{ \log^2 \left( \frac{n}{k}\right)}{\log \left( \frac{\lambda n}{k \log \left(\frac{n}{k} \right)} \right)}\right)=\omega\left( \frac{k^2}{\lambda^2 n} \frac{ \log^2 \left( \frac{n}{k}\right)}{\log \left( \frac{\lambda^2 n^2}{k^2 \log^2 \left(\frac{n}{k} \right)} \right)}\right)$$ 
which is equivalent with 
$$\frac{\lambda^2 n^2}{k^2 \log^2 \left(\frac{n}{k} \right)} \log \left( \frac{\lambda^2 n^2}{k^2 \log^2 \left(\frac{n}{k} \right)} \right) = \omega \left( \frac{n}{k'}\right).$$

Based on the calculus lemma \cite[Proposition 12(c)]{gamarnik2019landscape} this is now equivalent with 

$$\frac{\lambda^2 n^2}{k^2 \log^2 \left(\frac{n}{k} \right)}  = \omega \left( \frac{\frac{n}{k'}}{ \log \left( \frac{n}{k'} \right)}\right)$$ or

$$k'=\omega\left( \frac{k^2}{\lambda^2 n} \frac{ \log^2 \left( \frac{n}{k}\right)}{\log \left( \frac{n}{k'} \right)}\right) $$ or
$$\frac{\frac{n}{k}}{\log \left(\frac{n}{k} \right)}= \omega\left(\frac{1}{\lambda}\sqrt{\frac{ n}{k'  \log \left(\frac{n}{k'}\right)}}\right).$$Since $\lambda=o(1)=o\left(\sqrt{\frac{n}{k'  \log \left(\frac{n}{k'}\right)}} \right)$, by Assumption \ref{assum:LambdaMain}, we have by the calculus lemma \cite[Proposition 12(d)]{gamarnik2019landscape}  $$\frac{n}{k} =\omega \left(\frac{1}{\lambda}\sqrt{\frac{ n}{k'  \log \left(\frac{n}{k'}\right)}} \log \left(\frac{1}{\lambda}\sqrt{\frac{ n}{k'  \log \left(\frac{n}{k'}\right)}}\right) \right)$$
or switching to $o$-notation and using simple rearranging, $$\ell_c=\frac{1}{2\lambda}k\sqrt{\frac{k'}{n  \log \left(\frac{n}{k'}\right)}} \log \left( \frac{1}{2\lambda}\sqrt{\frac{n}{k'  \log \left(\frac{n}{k'}\right)}}  \right)=o\left(k'\right).$$

Second, by Assumptions \ref{assum:k'Main}, \ref{assum:LambdaMain} we know $$k'= o\left(n  \min\{\lambda^2 \log \left(\frac{1}{\lambda} \right) ,1\}  \right),$$which is equivalent with $$\min\{\frac{1}{\lambda^2} \log \left(\frac{1}{\lambda^2} \right) ,1\}=o\left(\frac{n}{k'}  \right),$$ which now since $\lambda=o(1)$ implies based on by the calculus lemma \cite[Proposition 12,(d)]{gamarnik2019landscape} we have also
$$ \frac{1}{\lambda^2} =o \left( \frac{n}{k' \log \left( \frac{n}{k'}\right)}\right),$$ or 
$$ \frac{1}{2\lambda}\sqrt{\frac{n}{k' \log \left( \frac{n}{k'}\right)}} =o \left( \frac{n}{k' \log \left( \frac{n}{k'}\right)}\right).$$By  \cite[Proposition 12,(d)]{gamarnik2019landscape}, the last displayed equation implies
$$\frac{1}{2\lambda}\sqrt{\frac{n}{k' \log \left( \frac{n}{k'}\right)}} \log \left( \frac{1}{2\lambda}\sqrt{\frac{n}{k' \log \left( \frac{n}{k'}\right)}} \right)=o\left( \frac{n}{k'}\right)$$or  
$$\frac{1}{2\lambda}k\sqrt{\frac{k'}{ n \log \left( \frac{n}{k'}\right)}} \log \left( \frac{1}{2\lambda}\sqrt{\frac{n}{k' \log \left( \frac{n}{k'}\right)}} \right)=o\left( k\right)$$which means
$$\ell_c=o\left(k\right),$$as desired.
Finally note that $$\frac{\ell_c}{\frac{k'k}{n}}=\frac{1}{2\lambda}\sqrt{\frac{n}{k'  \log \left(\frac{n}{k'}\right)}}\log \left( \frac{1}{2\lambda}\sqrt{\frac{n}{k'  \log \left(\frac{n}{k'}\right)}}  \right)$$which since $\lambda=o(1)$ and $k'=o(n)$ implies $\lambda=o\left(\sqrt{\frac{n}{k'  \log \left(\frac{n}{k'}\right)}} \right)$ we can conclude

$$\ell_c=\omega\left(\frac{kk'}{n}\right).$$
The proof of \eqref{ineqov} is complete.

We now turn to the second part which analyzed the monotonicity of the function. We have for every $\floor{kk'/n} \leq \ell \leq \min\{k,k'\}$, that $\Gamma_{k'}\left(\ell+1\right)-\Gamma_{k'}\left(\ell \right)$ equals
\begin{align}
    &-\lambda\frac{2\ell+1}{k}-2k'\sqrt{\frac{1}{n}}\left(\sqrt{  \log \left[ \binom{k}{\ell+1} \binom{n-k}{k'-\ell-1} \right]}-\sqrt{\log \left[ \binom{k}{\ell} \binom{n-k}{k'-\ell} \right]}\right) \notag\\
    &=-\lambda\frac{2\ell+1}{k}-2k'\sqrt{\frac{1}{n}} \frac{ \log \left[ \binom{k}{\ell+1} \binom{n-k}{k'-\ell-1} \right]-\log \left[ \binom{k}{\ell} \binom{n-k}{k'-\ell} \right]}{\sqrt{\log \left[ \binom{k}{\ell} \binom{n-k}{k'-\ell} \right]}+\sqrt{\log \left[ \binom{k}{\ell+1} \binom{n-k}{k'-\ell-1} \right]}} \notag\\
    &=-\lambda\frac{2\ell+1}{k}+2k'\sqrt{\frac{1}{n}} \frac{ \log \frac{\left(\ell+1\right) \left(n-k-k'+\ell+1\right)}{\left(k'-\ell\right)\left(k-\ell\right)}}{\sqrt{\log \left[ \binom{k}{\ell} \binom{n-k}{k'-\ell} \right]}+\sqrt{\log \left[ \binom{k}{\ell+1} \binom{n-k}{k'-\ell-1} \right]}} \label{diff1}
\end{align}
where in the last equality we used that from elementary algebra
\begin{align*}
    \log \left[ \binom{k}{\ell} \binom{n-k}{k'-\ell} \right]-\log \left[ \binom{k}{\ell+1} \binom{n-k}{k'-\ell-1} \right]=\log \frac{\left(\ell+1\right) \left(n-k-k'+\ell+1\right)}{\left(k'-\ell\right)\left(k-\ell\right)}.
\end{align*}

Now using the basic asymptotic identity that for $m_2=o(m_1),$ $\log \binom{m_1}{m_2}=(1+o(1))m_2 \log (m_1/m_2)$ as $m_1 \rightarrow +\infty$, we have since $k'=o(n)$ and for all $\ell$ of interest it holds $\ell=o(\min\{k,k'\})$,
\begin{align*}
\log \left[  \binom{k}{\ell} \binom{n-k}{k'-\ell}\right]&=(1+o(1))\left[ \ell \log \left(\frac{k}{\ell}\right) + (k'-\ell) \log \left( \frac{n-k}{k'-\ell} \right) \right]
\end{align*}which equals
\begin{align}
&=\left(1+o(1)\right) \left[  k' \log \left(\frac{n}{k'} \right)+k'\log\left(\frac{n-k}{n}\right)- \left(k'+\ell\right) \log \left(\frac{k'-\ell}{k'} \right)-\ell\log \left(\frac{\ell (n-k)}{k k'}\right)   \right] \notag\\
&=\left(1+o(1)\right) \left[  k' \log \left(\frac{n}{k'} \right) -\frac{k}{n}k' +\frac{\ell}{k'}\left(k'+\ell\right) -\ell\log \left(\frac{\ell (n-k)}{k k'}\right)   \right] \label{log} \\
&=\left(1+o(1)\right)  k' \log \left(\frac{n}{k'} \right)-O\left(\ell  \log \left(\frac{n-k}{k'}\right)  \right) \label{ok'0}\\
&=\left(1+o(1)\right)  k' \log \left(\frac{n}{k'} \right) \label{ok'}
\end{align}where for (\ref{log}) we used the basic $\log(1+x)=(1+o(1))x,$ as $x \rightarrow 0$, for \eqref{ok'0} we used $\ell=o(k')$, $\ell \leq \min\{k', k\}$ and $k'=o(n)$ and for (\ref{ok'}) we used that $\ell=o(k').$

Now plugging \eqref{ok'} in \eqref{diff1} we have that for every $\ell=o(k'),$
\begin{align*}
    \Gamma_{k'}\left(\ell+1\right)-\Gamma_{k'}\left(\ell \right)&=-\lambda\frac{2\ell+1}{k}+\left(1+o(1)\right)2k'\sqrt{\frac{1}{n}} \frac{ \log \frac{\left(\ell+1\right) \left(n-k-k'+\ell+1\right)}{\left(k'-\ell\right)\left(k-\ell\right)}}{2\sqrt{k' \log \left(\frac{n}{k'}\right)}} \\
    &=-\lambda\frac{2\ell+1}{k}+\left(1+o(1)\right)\sqrt{\frac{k'}{n  \log \left(\frac{n}{k'}\right)}} \log \frac{\left(\ell+1\right) \left(n-k-k'+\ell+1\right)}{\left(k'-\ell\right)\left(k-\ell\right) }\\
    \end{align*}which implies $\Gamma_{k'}\left(\ell+1\right)-\Gamma_{k'}\left(\ell \right)$ equals
     \begin{align}
     &-\lambda\frac{2\ell+1}{k}\notag\\
     &+\left(1+o(1)\right)\sqrt{\frac{k'}{n  \log \left(\frac{n}{k'}\right)}} \left[\log \frac{\left(\ell+1\right) n}{k' k }+\log \left(\frac{n-k-k'+\ell+1}{n}\right)-\log \left(\frac{k'-\ell}{k'}\right)-\log \left(\frac{k-\ell}{k}\right)\right]. \label{diff2}
    \end{align}Using that $ \ell \geq \floor{kk'/n}$ we have $$\frac{\left(\ell+1\right) n}{k' k }=\Omega\left(1\right).$$ Furthermore using the basic $\log(1+x)=(1+o(1))x,$ as $x \rightarrow 0$  and $k=o(n), \ell=o(k'),\ell=o(k)$ we have $$\log \left(\frac{n-k-k'+\ell+1}{n}\right)-\log \left(\frac{k'-\ell}{k'}\right)-\log \left(\frac{k-\ell}{k}\right)=o(1).$$ Hence from \eqref{diff2} we conclude that for all  $  \floor{kk'/n} \leq \ell =o(\min\{k',k\}),$
    \begin{align}
    \Gamma_{k'}\left(\ell+1\right)-\Gamma_{k'}\left(\ell \right)
     &=-\lambda\frac{2\ell+1}{k}+\left(1+o(1)\right)\sqrt{\frac{k'}{n  \log \left(\frac{n}{k'}\right)}} \log \left(\frac{\left(\ell+1\right) n}{k' k }\right) \notag\\
     &=-2\lambda\frac{\ell+1}{k}+\frac{\lambda}{k}+\left(1+o(1)\right)\sqrt{\frac{k'}{n  \log \left(\frac{n}{k'}\right)}} \log \left(\frac{\left(\ell+1\right) n}{k' k }\right) \label{Diff2a}
     \end{align}
    which implies that $\Gamma_{k'}\left(\ell+1\right)-\Gamma_{k'}\left(\ell \right)$ equals to
    \begin{align}
    &-\frac{2\lambda k'}{n}\log \left(\frac{\left(\ell+1\right) n}{k' k }\right) \left[ \frac{\frac{\left(\ell+1\right) n}{k' k }}{\log \left(\frac{\left(\ell+1\right) n}{k' k }\right)}  -\left(1+o(1)\right)\frac{1}{2\lambda}\sqrt{\frac{n}{k'  \log \left(\frac{n}{k'}\right)}}\right]+\frac{\lambda}{k} \notag\\
     =&- \sqrt{\frac{ k'}{n  \log \left(\frac{n}{k'}\right)}} \log \left(\frac{\left(\ell+1\right) n}{k' k }\right) \left[ \frac{\frac{\left(\ell+1\right) n}{k' k }}{\log \left(\frac{\left(\ell+1\right) n}{k' k }\right)}  /\left[\frac{1}{2\lambda}\sqrt{\frac{n}{k'  \log \left(\frac{n}{k'}\right)}}\right]-\left(1+o(1)\right)\right]+\frac{\lambda}{k}. \label{diffFin}
    \end{align}
    
    We now prove part (a). Assume $\floor{kk'/n} \leq \ell \leq \floor{(1-\delta) \ell_{c}}-1$ for some fixed $\delta>0.$ Recall that since $\lambda=o(1)$ and $k'=o(n)$ it holds,$$\lambda=o\left(\sqrt{\frac{n}{k'  \log \left(\frac{n}{k'}\right)}} \right).$$ Hence, by Lemma \ref{lem:calc} we have $$\limsup_n \frac{\frac{\left(\ell+1\right) n}{k' k }}{\log \left(\frac{\left(\ell+1\right) n}{k' k }\right)}  /\left[\frac{1}{2\lambda}\sqrt{\frac{n}{k'  \log \left(\frac{n}{k'}\right)}}\right]$$ equals to $$\limsup_n \frac{\left(\ell+1\right) n}{k' k } /\left[\frac{1}{2\lambda}\sqrt{\frac{n}{k'  \log \left(\frac{n}{k'}\right)}} \log \left(\frac{1}{2\lambda}\sqrt{\frac{n}{k'  \log \left(\frac{n}{k'}\right)}} \right)\right]$$or $$\limsup_n \left(\ell+1\right) /\left[\frac{1}{2\lambda}k\sqrt{\frac{k'}{n  \log \left(\frac{n}{k'}\right)}} \log \left( \frac{1}{2\lambda}\sqrt{\frac{n}{k'  \log \left(\frac{n}{k'}\right)}}  \right)\right]$$which by definition of $\ell_c$ equals 
    $$\limsup_n \left(\ell+1\right) / \ell_c $$which by assumption in this case is at most $1-\delta.$
    
    Hence, for any $\epsilon>0$, for large enough values of $n$,
    \begin{align*}
        \frac{\frac{\left(\ell+1\right) n}{k' k }}{\log \left(\frac{\left(\ell+1\right) n}{k' k }\right)}  /\left[\frac{1}{2\lambda}\sqrt{\frac{n}{k'  \log \left(\frac{n}{k'}\right)}}\right] \leq 1-\delta (1-\frac{\epsilon}{2}),
    \end{align*}which based on \eqref{diffFin} gives for large enough values of $n$,
     \begin{align*}
    \Gamma_{k'}\left(\ell+1\right)-\Gamma_{k'}\left(\ell \right) &\geq - \sqrt{\frac{ k'}{n  \log \left(\frac{n}{k'}\right)}} \log \left(\frac{\left(\ell+1\right) n}{k' k }\right) \left[1-\delta (1-\frac{\epsilon}{2})-\left(1-\delta\frac{\epsilon}{2}\right)\right]+\frac{\lambda}{k}\\
    & =  \delta \left(1-\epsilon\right)  \sqrt{\frac{ k'}{n  \log \left(\frac{n}{k'}\right)}} \log \left(\frac{\left(\ell+1\right) n}{k' k }\right)+\frac{\lambda}{k}\\
    & \geq \delta \left(1-\epsilon\right)  \sqrt{\frac{ k'}{n  \log \left(\frac{n}{k'}\right)}} \log \left(\frac{\left(\ell+1\right) n}{k' k }\right).
    \end{align*}This completes the proof of part (a).
    
   Part (b) follows directly from \eqref{Diff2a} and the fact that $ \log \left(\frac{\left(\ell+1\right) n}{k' k }\right) \leq \log \left(\frac{ n}{k' }\right)$ since $\ell+1 \leq k.$

We now turn to part (c). First notice that under both Assumption \ref{assum:LambdaMain} and Assumption \ref{assum:k'Main} we have
    \begin{align}\label{k'LB}
    k'=\omega\left(\frac{k^2}{\lambda^2 n} \frac{\log^2 \left(\frac{n}{k}\right) }{\log \left(\frac{\lambda n}{k \log \frac{n}{k}}\right) }\right).\end{align}
  Hence,
    $$\frac{1}{\lambda}k\sqrt{\frac{k'}{n  \log \left(\frac{n}{k'}\right)}} \log \left( \frac{1}{\lambda}\sqrt{\frac{n}{k'  \log \left(\frac{n}{k'}\right)}}  \right)  \geq \frac{k^2}{n \lambda^2 } \sqrt{\frac{\log^2 \left(\frac{n}{k}\right) }{\log \left(\frac{n}{k'}\right) \log \left(\frac{\lambda n}{k \log \frac{n}{k}}\right) }} \log \left( \frac{1}{\lambda}\sqrt{\frac{n}{k'  \log \left(\frac{n}{k'}\right)}}  \right)  $$which since $\lambda=o(1)$ and $k,k'=o(n)$ implies 
    $$\ell_c=\omega \left(\frac{k^2}{n \lambda^2}\right).$$Since $\lambda=o\left(\frac{k}{\sqrt{n}} \right)$ we conclude
    \begin{align}\label{eq:lc_omega1}
       \ell_c=\omega\left(1\right).
    \end{align}

     Now for all $\ell \geq 10\ceil{\ell_c}-1$ it holds for sufficiently large $n$ that \begin{align}\label{eq:finaldiff1}
          \lambda \frac{\ell+1}{2k} \geq \frac{\lambda}{k},
      \end{align}because of \eqref{eq:lc_omega1}. Furthermore, for all $\ell \geq 10\ceil{\ell_c}-1$ it also holds for sufficiently large $n$ that, \begin{align}\label{eq:finaldiff2}
          \lambda \frac{\ell+1}{2k} \geq \left[\frac{1}{2}\sqrt{\frac{k'}{n  \log \left(\frac{n}{k'}\right)}}  \log \left(\frac{\left(\ell+1\right)n}{k'k} \right)\right].
      \end{align}To see this, notice that by simple algebra
      $$\liminf_n \left[\lambda \frac{\ell+1}{2k} \right]/\left[\frac{1}{2}\sqrt{\frac{k'}{n  \log \left(\frac{n}{k'}\right)}}  \log \left(\frac{\left(\ell+1\right)n}{k'k} \right)\right]$$ equals  $$\liminf_n \frac{\frac{\left(\ell+1\right) n}{k' k }}{\log \left(\frac{\left(\ell+1\right) n}{k' k }\right)}  /\left[\frac{1}{\lambda}\sqrt{\frac{n}{k'  \log \left(\frac{n}{k'}\right)}}\right].$$
      Similarly to the proof of Part 1, by Lemma \ref{lem:calc} we have that it equals to $$\liminf_n \left(\ell+1\right) / 2\ell_c $$which by assumption in this case is at least $10/2=5>1,$ implying \eqref{eq:finaldiff2} for sufficiently large $n$. Combining \eqref{eq:finaldiff1}, \eqref{eq:finaldiff2} with \eqref{Diff2a}, we conclude for all $\ell \geq 10 \ceil{\ell_c}-1$ and sufficiently large $n$,
      
      \begin{align}
    \Gamma_{k'}\left(\ell+1\right)-\Gamma_{k'}\left(\ell \right) \leq -\lambda  \frac{\ell+1}{k},
    \end{align}as we wanted. This completes the proof of the part (c).

      We now proceed with the proof of the part (d). Recall that in this setting $0 \leq \delta_0 \leq \delta \leq \delta_1<1$. Notice that using the first part and telescopic summation we conclude \begin{align}\label{productGamma}
    \Gamma_{k'}\left(\floor{\left(1-\delta\right)\ell_c}\right)-\Gamma_{k'}\left(\floor{kk'/n} \right) \geq \delta \left(1-\epsilon\right)\sqrt{\frac{k'}{n \log \frac{n}{k'}}} \log \left( \prod_{\ell=\floor{kk'/n}}^{\floor{\left(1-\delta\right)\ell_c}} \frac{\left(\ell+1\right)n}{k'k}\right)
    \end{align}Now notice $\frac{\left(\ell+1\right)n}{k'k} \geq 1$ for all $\ell$ of interest. Furthermore, since $\ell_c=\omega\left(\frac{k'k}{n}\right)$ and $\ell_c=\omega\left(1\right)$, focusing on the $\ell$ who satisfy $ \frac{\floor{\left(1-\delta\right)\ell_c}}{2} \leq \ell \leq \floor{\left(1-\delta\right)\ell_c}$ we know that first there are at least $\frac{\left(1-\delta\right)\ell_c}{2}$ such values of $\ell$, and second, for sufficiently large $n$, $\frac{\left(\ell+1\right)n}{k'k} \geq e$. We conclude that for sufficiently large $n$, $$\prod_{\ell=\floor{kk'/n}}^{\floor{\left(1-\delta\right)\ell_c}} \frac{\left(\ell+1\right)n}{k'k} \geq e^{\frac{\left(1-\delta\right)\ell_c}{2}}.$$
    Hence using \eqref{productGamma},
     \begin{align*}
    \Gamma_{k'}\left(\floor{\left(1-\delta\right)\ell_c}\right)-\Gamma_{k'}\left(\floor{kk'/n} \right) \geq \frac{\left(1-\delta\right)\ell_c}{2}\delta \left(1-\epsilon\right)\sqrt{\frac{k'}{2n \log \frac{n}{k'}}}
    \end{align*}or using that $\delta_0 \leq \delta \leq \delta_1$ and choosing $\epsilon=\frac{1}{2}$ gives,
    \begin{align*}
    \Gamma_{k'}\left(\floor{\left(1-\delta\right)\ell_c}\right)-\Gamma_{k'}\left(\floor{kk'/n} \right) \geq \frac{\left(1-\delta_1\right)\ell_c}{4}\delta_0 \sqrt{\frac{k'}{2n \log \frac{n}{k'}}}
    \end{align*}
    Now since $\lambda=o(1)$ we have from the definition of $\ell_c$ \begin{align}\label{ellclog} \ell_c=\Omega \left( \frac{1}{\lambda}k\sqrt{\frac{k'}{n  \log \left(\frac{n}{k'}\right)}} \log \left( \sqrt{\frac{n}{k'  \log \left(\frac{n}{k'}\right)}}  \right) \right)=\Omega \left( \frac{1}{\lambda}k\sqrt{ \frac{k' \log \frac{n}{k'} }{n}}\right).\end{align} In particular,
    \begin{align}\label{eq:4part_1}
    \Gamma_{k'}\left(\floor{\left(1-\delta\right)\ell_c}\right)-\Gamma_{k'}\left(\floor{kk'/n} \right) \geq D \frac{k'k}{\lambda n},
    \end{align}for some $D$ which besides absolute constants depends only on $\delta_0,\delta_1.$
    
    Now, for $C>10$ it holds \begin{align*}
    \Gamma_{k'}\left(C\ceil{\ell_c}\right)-\Gamma_{k'}\left(\floor{\left(1-\delta\right)\ell_c}\right)&= \left(\Gamma_{k'}\left(C\ceil{\ell_c}\right)-\Gamma_{k'}\left(10\ceil{\ell_c}\right) \right)+\left(\Gamma_{k'}\left(10\ceil{\ell_c}\right)-\Gamma_{k'}\left(\floor{\left(1-\delta\right)\ell_c}\right)\right).
    \end{align*}
    Hence using the Parts (b), (c) of the current Lemma to bound the first and second summand we conclude  \begin{align}
    \Gamma_{k'}\left(C\ceil{\ell_c}\right)-\Gamma_{k'}\left(\floor{\left(1-\delta\right)\ell_c}\right) 
    & \leq -\lambda \frac{1}{k}\sum_{\ell=10\ceil{\ell_c}}^{C\ceil{\ell_c}} \ell+\sum_{\ell= \floor{\left(1-\delta\right)\ell_c}}^{10\ceil{\ell_c}}(1+\epsilon)\sqrt{ \frac{k' \log \frac{n}{k'} }{n}}\notag\\
    & \leq -\frac{\lambda }{k}10C \ell_c^2+O\left( \ell_c \sqrt{ \frac{k' \log \frac{n}{k'} }{n}} \right), \label{eq:sumSplit}
    \end{align} where for the last inequality we used elementary algebra.

    Now using \eqref{ellclog} and \eqref{eq:sumSplit} we conclude for some sufficiently large constant $C>0$
    that 
    \begin{align*}
    \Gamma_{k'}\left(C\ceil{\ell_c}\right)-\Gamma_{k'}\left(\floor{\left(1-\delta\right)\ell_c}\right) &\leq -5C\frac{\lambda}{k} \ell_c^2
    \end{align*}
     Using \eqref{ellclog} once again we conclude that
     \begin{align}\label{eq:4part_2}
        \Gamma_{k'}\left(\floor{\left(1-\delta\right)\ell_c}\right)-\Gamma_{k'}\left(C\ceil{\ell_c} \right) =\Omega\left(\frac{k'k}{\lambda n} \right),
    \end{align}as we wanted.
     Combining \eqref{eq:4part_1} and \eqref{eq:4part_2} yields the final result. This completes the proof of part (d) and the proof of the theorem.
\end{proof}

\section{Proof of Theorem \ref{FM}, Part 1 (First Moment Method)}\label{sec:1MM}
\begin{proof}[Proof of Theorem \ref{FM}, part (1)]We fix some $\ell=\floor*{kk'/n},\floor*{kk'/n}+1,\ldots,k'$.
Let us define
\begin{align}\label{Aell}
    A_{\ell}:= -\lambda \frac{\ell^2}{k}- 2k'\sqrt{\frac{1}{n} \left(\log \left[ \binom{k}{\ell} \binom{n-k}{k'-\ell}\right]+\alpha_n \right)}.
\end{align}By a union bound we have
\begin{align}
    \mathbb{P}\left(\phi_{k'}(\ell)  \leq A_{\ell} \right) &=\mathbb{P}\left(\bigcup_{v \in \{0,1\}^n: \|v\|_0=k',  \langle v,x \rangle=\ell} \{H(v) \leq A_{\ell}\}\right) \notag\\
    & \leq \binom{k}{\ell}\binom{n-k}{k'-\ell} \mathbb{P}\left(H(v) \leq A_{\ell} \right), \label{FM:UB}
\end{align}
where in the last line $v$ is chosen arbitrarily from $ v \in \{0,1\}^n, \|v\|_0=k',  \langle v,x \rangle=\ell$ since for any such $v$ the value of $\mathbb{P}\left(H(v) \leq A_{\ell} \right)$ remains the same. More specifically, notice that for any such $v$, $v^\top W v$ is distributed as $N\left(0,\frac{2(k')^2}{n}\right).$ Hence, using that for any such $v$, $$H(v)=-v^\top Y v=-\frac{\lambda}{k}\langle v,x\rangle^2-v^\top W v=-\frac{\lambda}{k}\ell^2-v^\top W v$$we conclude that since $W$ is a GOE($n$), $$H(v) \sim N\left(-\frac{\lambda}{k}\ell^2,\frac{2(k')^2}{n}\right),$$ which is a distribution which is independent of the specific choice of $v$.

Furthermore, combining the last displayed equality in distribution with (\ref{Aell}) we conclude
\begin{align}
     \mathbb{P}\left(H(v) \leq A_{\ell} \right) &= \mathbb{P}\left(Z \geq 2k'\sqrt{\frac{1}{n} \left(\log \left[ \binom{k}{\ell} \binom{n-k}{k'-\ell}\right]+\alpha_n \right)}\right),\\
     &\qquad\text{ for } Z \sim N\left(0,\frac{2(k')^2}{n}\right)\\
     &= \mathbb{P}\left(Z \geq \sqrt{ 2\left(\log \left[ \binom{k}{\ell} \binom{n-k}{k'-\ell}\right]+\alpha_n \right)}\right), \text{ for } Z \sim N\left(0,1\right)\\
     & \leq \frac{1}{ \sqrt{ 2\left(\log \left[ \binom{k}{\ell} \binom{n-k}{k'-\ell}\right]+\alpha_n \right)}} \exp \left(-  \left(\log \left[ \binom{k}{\ell} \binom{n-k}{k'-\ell}\right]+\alpha_n \right)\right), \label{Mills}\\
     &=\frac{1}{  \binom{k}{\ell} \binom{n-k}{k'-\ell}\sqrt{ 2\left(\log \left[ \binom{k}{\ell} \binom{n-k}{k'-\ell}\right]+\alpha_n \right)}}\exp\left(-\alpha_n\right),
\end{align}
where for (\ref{Mills}) we used the standard Mill's ratio upper bound.
Using (\ref{FM:UB}) we have for every $\ell,$
\begin{align}
    \mathbb{P}\left(\phi_{k'}(\ell)  \leq A_{\ell} \right) \leq \frac{1}{ \sqrt{ 2\left(\log \left[ \binom{k}{\ell} \binom{n-k}{k'-\ell}\right]+\alpha_n \right)}}\exp\left(-\alpha_n\right) \leq \frac{1}{\sqrt{\alpha_n} } \exp\left(-\alpha_n\right).
\end{align}

We now use a union bound over the possible values of $\ell$,
\begin{align} \label{FM:UB2}
    \mathbb{P}\left(\bigcup_{\ell=\floor*{kk'/n},\floor*{kk'/n}+1,\ldots,k'} \{\phi_{k'}(\ell)  \leq A_{\ell}\} \right) \leq \frac{k'}{ \sqrt{ \alpha_n}}\exp\left(-\alpha_n\right). 
\end{align}Combining (\ref{FM:UB2}) with our assumed $\alpha_n=\omega\left(\log k'\right)$ we conclude 
\begin{align}
    \lim_{n \rightarrow +\infty} \mathbb{P}\left(\bigcup_{\ell=\floor*{kk'/n},\floor*{kk'/n}+1,\ldots,k'} \{\phi_{k'}(\ell)  \leq A_{\ell}\} \right) =0.
\end{align}This completes the proof of the theorem.
\end{proof}

\section{Proof of Theorem \ref{FM}, Part 2 (Second Moment Method)}\label{sec:2MM}

\begin{proof}[Proof of Theorem \ref{FM}, part (2)]

We start with the following observation. For any $\ell$ if $v \in \{0,1\}^n$ with $\|v\|_0=k$ and $\inner{v,x}=\ell$ then $-v^{\top}Yv=-\lambda \frac{\ell^2}{k}-v^{\top}Wv$. Hence it holds $$-\phi_{k'}(\ell)=-\lambda \frac{\ell^2}{k}+\psi_{k'}(\ell)$$ where \begin{align}\label{def:psiK}\psi_{k'}(\ell)=\max_{v \in \{0,1\}^n: \|v\|_0=k, \inner{v,x}=\ell} v^{\top}Wv.\end{align} In particular, to show our result it suffices that there exist a universal constants $c,C>0$ such that the following holds.   If $\ell \leq c \min \{k',k\}$ then it holds

\begin{align}\label{psiK}
2k'\sqrt{\frac{1}{n}\log \left[ \binom{k}{\ell}\binom{n-k}{k'-\ell}\right]} -C\sqrt{\frac{k'\log n}{n} \max\{\frac{ \ell^4}{k^2}, \frac{(k')^4}{n^2},k'\}} 
\leq \psi_{k'}(\ell)
\end{align}
with high probability as $n \rightarrow + \infty.$ We proceed towards proving that \eqref{psiK} holds for all $\ell \leq c \min \{k',k\}$.

We fix some $\ell$ and assume for some sufficiently small constant $c>0$ that $\ell \leq c \min\{k',k\}.$ We derive lower bounds on the probability that \eqref{psiK} holds by using the second moment method. The argument is completed with a union bound over $\ell$. The constant $c$ is assumed to be sufficiently small to guarantee the validity of the proof steps.

We define $$T_{\ell}:=\{ v \in \{0,1\}^n: \|v\|_0=k', \inner{v,x}=\ell\},$$ $$\mathcal{Z}_{\ell}:=\{ v \in T_{\ell}:  v^{\top}Wv \geq  2k'\sqrt{\frac{1}{n}\log \left[ \binom{k}{\ell}\binom{n}{k'-\ell}\right]}\},$$and $$Z_{\ell}:=|\mathcal{Z}_{\ell}|.$$ By linearity of expectation since $|T_{\ell}|=\binom{k}{\ell}\binom{n}{k'-\ell}$,
$$\E{Z_{\ell}}=\binom{k}{\ell}\binom{n}{k'-\ell}\bP \left(v^{\top}Wv \geq  2k'\sqrt{\frac{1}{n}\log \left[ \binom{k}{\ell}\binom{n}{k'-\ell}\right]}\right),$$where $v$ is arbitrary vector with  $v \in \{0,1\}^n, \|v\|_0=k', \inner{v,x}=\ell.$ Since $W$ is GOE$(n)$, it can be straightforwardly checked that $v^{\top}Wv \sim N\left(0,\frac{2(k')^2}{N}\right).$ Hence,
\begin{align}\label{FMe}
\E{Z_{\ell}}=\binom{k}{\ell}\binom{n}{k'-\ell}\bP \left(Z \geq \sqrt{2 \log \left[ \binom{k}{\ell}\binom{n}{k'-\ell}\right]}\right),
\end{align}
where $Z \sim N\left(0,1\right).$ 

Now \begin{align}
    Z_{\ell}^2&=|\{ v,u \in T_{\ell} : , \min\{v^{\top}Wv,u^{\top}Wu \} \geq 2k'\sqrt{\frac{1}{n}\log \left[ \binom{k}{\ell}\binom{n-k}{k'-\ell}\right]}\}|\\
    = &\sum_{m=0}^{k'} |\{ v,u \in T_{\ell}: \inner{v,u}=m, \min\{v^{\top}Wv,u^{\top}Wu \}\geq  2k'\sqrt{\frac{1}{n}\log \left[ \binom{k}{\ell}\binom{n}{k'-\ell}\right]}\}|,
\end{align}where for the last equality we decomposed the pairs $u,v$ based on the possible values of $\inner{u,v}$. Now using linearity of expectation using an argument similar to the argument for the first moment of $Z$,  $$\mathbb{E}[Z^2]=\sum_{m=0}^{k'} |\{ v,u \in T_{\ell}: \inner{v,u}=m\}|\,\mathbb{P}\left(Z_{1,\ell},Z_{2,\ell} \geq  \sqrt{2\log \left[ \binom{k}{\ell}\binom{n-k}{k'-\ell}\right]}\right)$$ where for each $m=0,1,2,\ldots,k'$, $Z_{1,m} \sim N\left(0,1\right),Z_{2,m} \sim N\left(0,1\right)$ and $\mathrm{Cov}(Z_{1,m},Z_{2,m})=\frac{m^2}{\left(k'\right)^2}$.

It holds

 \begin{align}
\frac{\mathbb{E}[Z^2]}{\E{Z}^2}=&\sum_{m=0}^{k'}  \frac{|\{ v,u \in T_{\ell}: \inner{v,u}=m\}|}{\left[\binom{k}{\ell}\binom{n-k}{k'-\ell}\right]^2} \frac{\mathbb{P}(Z_{1,m},Z_{2,m} \geq  \sqrt{2\log \left[ \binom{k}{\ell}\binom{n-k}{k'-\ell}\right]}\})}{\mathbb{P}(Z\geq  \sqrt{2\log \left[ \binom{k}{\ell}\binom{n-k}{k'-\ell}\right]})^2}.\\
=&\sum_{m=0}^{k'-1}  \frac{|\{ v,u \in T_{\ell}: \inner{v,u}=m\}|}{\left[\binom{k}{\ell}\binom{n-k}{k'-\ell}\right]^2} \frac{\mathbb{P}(Z_{1,m},Z_{2,m} \geq  \sqrt{2\log \left[ \binom{k}{\ell}\binom{n-k}{k'-\ell}\right]}\})}{\mathbb{P}(Z\geq  \sqrt{2\log \left[ \binom{k}{\ell}\binom{n-k}{k'-\ell}\right]})^2}\\
&+\frac{1}{\binom{k}{\ell}\binom{n}{k'-\ell}\bP \left(Z \geq \sqrt{2 \log \left[ \binom{k}{\ell}\binom{n}{k'-\ell}\right]}\right)},\label{ratio0} 
\end{align}
where $Z_{1,m},Z_{2,m}$ are defined above, $Z$ is an independent standard normal variable and for the last equality we used that $\{ v,u \in T_{\ell}: \inner{v,u}=k'\}=\{ v,u \in T_{\ell}: u=v\}$, $|T_{\ell}|=\binom{k}{\ell}\binom{n-k}{k'-\ell}$ and $Z_{1,k'}=Z_{2,k'}$ almost surely as standard normal random variables with correlation one. 

Using the standard Mills ratio lower bound; for $x \geq 1,$ \begin{align} \mathbb{P}\left(Z \geq x\right) \geq \frac{1}{2\sqrt{2 \pi} x}\exp \left(-\frac{x^2}{2} \right),\end{align} since we may assume $\ell<k'/2<k'-1$ we have $\sqrt{\log \left[ \binom{k}{\ell}\binom{n-k}{k'-\ell}\right]} \geq 1$ and therefore
\begin{align*}
\mathbb{P}\left(Z\geq  \sqrt{2\log \left[ \binom{k}{\ell}\binom{n-k}{k'-\ell} \right]}\right) \geq  \frac{1}{4 \sqrt{\pi}\sqrt{\log \binom{k}{\ell}\binom{n-k}{k'-\ell}}} \exp \left( -\log \binom{k}{\ell}\binom{n-k}{k'-\ell} \right)
\end{align*}which implies
\begin{align}
\binom{k}{\ell}\binom{n-k}{k'-\ell}\mathbb{P}&\left(Z\geq  \sqrt{2\log \left[ \binom{k}{\ell}\binom{n-k}{k'-\ell}\right]}\right)\\
& \geq \binom{k}{\ell}\binom{n-k}{k'-\ell} \frac{1}{4 \sqrt{\pi}\sqrt{\log \binom{k}{\ell}\binom{n-k}{k'-\ell}}} \exp \left( -\log \binom{k}{\ell}\binom{n-k}{k'-\ell} \right) \notag\\
&=\frac{1}{4 \sqrt{\pi}\sqrt{\log \binom{k}{\ell}\binom{n-k}{k'-\ell}}}. \label{FMMills}
\end{align}

Hence, \eqref{ratio0} using \eqref{FMMills} implies 

\begin{align}
\frac{\mathbb{E}[Z^2]}{\E{Z}^2}& \leq \sum_{m=0}^{k'-1} \frac{|\{ v,u \in T_{\ell}: \inner{v,u}=m\}|}{\left[\binom{k}{\ell}\binom{n-k}{k'-\ell}\right]^2} \frac{\mathbb{P}(Z_{1,m},Z_{2,m} \geq  \sqrt{2\log \left[ \binom{k}{\ell}\binom{n-k}{k'-\ell}\right]}\})}{\mathbb{P}(Z\geq  \sqrt{2\log \left[ \binom{k}{\ell}\binom{n-k}{k'-\ell}\right]})^2}\\
&+4 \sqrt{\pi}\sqrt{\log \binom{k}{\ell}\binom{n-k}{k'-\ell}}.\label{ratio} 
\end{align}

Using now Lemma 4.2 from \cite{OGP-submatrix} we have for every $m=0,1,\ldots,k'-1$, and $\gamma_{m}=\sqrt{\mathrm{Cov}(Z_{1,m},Z_{2,m})}=\frac{m}{k'},$

\begin{align} 
\mathbb{P}&\left(Z_{1,m}, Z_{2,m} \geq \sqrt{2\log \left[ \binom{k}{\ell}\binom{n-k}{k'-\ell}\right]}\right)\\
&\qquad\leq \frac{\left(1+\gamma_{m}^2\right)^2}{2\pi \sqrt{1-\gamma_{m}^4}} \frac{1}{2 \log \binom{k}{\ell}\binom{n-k}{k'-\ell}}\exp\left(- \frac{2 \log \binom{k}{\ell}\binom{n-k}{k'-\ell}}{1+\gamma_{m}^2} \right)\\
&\qquad =\frac{\left(1+\gamma_{m}^2\right)^{\frac{3}{2}}}{4\pi \sqrt{1-\gamma^2_{m}}} \frac{1}{ \log \binom{k}{m}\binom{n-k}{k'-\ell}}\exp\left(- \frac{2 \log \binom{k}{\ell}\binom{n-k}{k'-\ell}}{1+\gamma_{m}^2} \right)\\
&\qquad \leq \frac{1}{\pi \sqrt{1-\gamma^2_{m}}} \frac{1}{ \log \binom{k}{\ell}\binom{n-k}{k'-\ell}}\exp\left(- \frac{2 \log \binom{k}{\ell}\binom{n-k}{k'-\ell}}{1+\gamma_{m}^2} \right)\label{ProbBound2} 
\end{align}where for the last inequality we used that $\gamma_{m} \leq 1.$

Combining \eqref{ProbBound2} with \eqref{ratio} we conclude

\begin{align} \label{ratio22}\frac{\mathbb{E}[Z^2]}{\E{Z}^2}-&4 \pi\sqrt{\log \binom{k}{\ell}\binom{n-k}{k'-\ell}} \\
&\leq \sum_{m=0}^{k'-1} \frac{|\{ v,u \in \mathcal{Z}_{\ell}: \inner{v,u}=m\}|}{\left[\binom{k}{\ell}\binom{n-k}{k'-\ell}\right]^2} \frac{16}{\sqrt{1-\gamma^2_{m}}} \exp\left( 2 \log \binom{k}{\ell}\binom{n-k}{k'-\ell} \frac{\gamma^2_{m}}{1+\gamma_{m}^2} \right).
\end{align}where using the exact value of $\gamma_m$ the last quantity equals to
\begin{align}
  \sum_{m=0}^{k'-1} \underbrace{\frac{|\{ v,u \in \mathcal{Z}_{\ell}: \inner{v,u}=m\}|}{\left[\binom{k}{\ell}\binom{n-k}{k'-\ell}\right]^2} \frac{16}{\sqrt{1-\frac{m^2}{(k')^2}}} \exp\left(  \log \binom{k}{\ell}\binom{n-k}{k'-\ell} \frac{2m^2}{m^2+(k')^2} \right)}_{A_{m}}. \label{ratio2}
\end{align}We now proceed to upper bound $\sum_{m=0}^{k'-1}A_m$. Let \begin{align}\label{mc} 
m_c:=\max\left\{\left\lceil\frac{32 \ell^2}{k}\right\rceil, \left\lceil\frac{16 (k')^2}{n} \right\rceil \right\}.
\end{align}

From Lemma \ref{comb:lemma}, part (a)  we have that for some sufficiently small constant $c>0$ if $\ell \leq c \min\{k',k\}$ then for sufficiently large values of $n$ it holds \begin{align}\label{mcok'} m_c \leq \frac{k'}{2}.\end{align}

We now distinguish two cases based on whether $m \leq m_c$ or $m > m_c.$ We first focus in the case where $m \geq m_c.$ For every $m=m_c,m_c+1,\ldots,k'-2$ we have via elementary algebra,
\begin{align}\label{Al}
\frac{A_{m+1}}{A_{m}}&=\frac{ |\{ v,u \in T_{\ell}: \inner{v,u}=m+1\}|}{|\{ v,u \in T_{\ell}: \inner{v,u}=m\}|} \sqrt{\frac{(k')^2-m^2}{(k')^2-(m+1)^2}}\\
&\times \exp\left[\frac{4m+2}{\left( m^2+\left(k'\right)^2\right)\left(\left(m+1\right)^2+\left(k'\right)^2\right)}\log \left[\binom{k}{\ell}\binom{n-k}{k'-\ell}\right]\right].
\end{align}

Using Lemma \ref{comb:lemma} we have for all $m \geq m_c,$

\begin{align}\label{ratioBin22}
    \frac{ |\{ v,u \in T_{\ell}: \inner{v,u}=m+1\}|}{|\{ v,u \in T_{\ell}: \inner{v,u}=m\}|} \leq \frac{1}{2}.
\end{align}

Now notice that $$\sqrt{\frac{(k')^2-m^2}{(k')^2-(m+1)^2}}$$is an increasing function of $m$ and therefore for all $m=0,1,2,\ldots,k'-2,$
\begin{align} \sqrt{\frac{(k')^2-m^2}{(k')^2-(m+1)^2}} &\leq \sqrt{\frac{(k')^2-(k'-2)^2}{(k')^2-(k'-1)^2}}\notag\\
&=\sqrt{\frac{4k'-4 }{2k'-1} } \notag\\
&\leq \sqrt{2}.\label{Term1}
\end{align}

Second notice that using the crude bound $\binom{k}{\ell}\binom{n-k}{k'-\ell} \leq \binom{n}{k'} \leq n^{k'}$ for all $m=0,1,2,\ldots,k'-2,$
\begin{align*}
  &\frac{4m+2}{\left( m^2+\left(k'\right)^2\right)\left(\left(m+1\right)^2+\left(k'\right)^2\right)}\log \left[\binom{k}{\ell}\binom{n-k}{k'-\ell}\right] \\
  &\qquad\leq  \frac{4m+2}{\left(k'\right)^4 }\log n^{k'} \\
  &\qquad\leq \frac{4(k'+1)}{\left(k'\right)^3} \log n \\
   &\qquad =o(1) 
\end{align*} where the last equality used the assumption $k'>\log n=\omega\left(\sqrt{ \log n}\right).$ Hence for large values of $n$, for every $m=0,1,2,\ldots,k'-2,$
we have
\begin{align}\label{Term2} \exp\left[\frac{4m+2}{\left( m^2+\left(k'\right)^2\right)\left(\left(m+1\right)^2+\left(k'\right)^2\right)}\log \left[\binom{k}{\ell}\binom{n-k}{k'-\ell}\right]\right] \leq \sqrt{2}.
\end{align}

Combining \eqref{ratioBin22}, \eqref{Term1}, \eqref{Term2} with \eqref{Al} we have for every $m=m_c, m_c+1,\ldots,k'-2$

\begin{align*}
\frac{A_{m+1}}{A_{m}} \leq \frac{1}{4} \sqrt{2} \sqrt{2}=\frac{1}{2}.
\end{align*}

Hence, $$\sum_{m=m_c}^{k'-1} A_{m} \leq \sum_{i=0}^{k'-m_c} \frac{A_{m_c}}{2^i}   \leq \sum_{i=0}^{+ \infty} \frac{A_{m_c}}{2^i} = 2A_{m_c}.$$

In particular, using \eqref{ratio22}, \eqref{ratio2} we have 
\begin{align}
    \frac{\mathbb{E}[Z^2]}{\E{Z}^2}&  \leq \sum_{m=0}^{m_c} A_{m}+2A_{m_c}+4 \pi\sqrt{\log \left[\binom{k}{\ell}\binom{n-k}{k'-\ell}\right]} \notag\\
    &\leq 2 \sum_{\ell=0}^{m_c} A_{m}+4 \pi\sqrt{\log \left[\binom{k}{\ell}\binom{n-k}{k'-\ell}\right]} \notag\\
    &\leq  2 \left(1+m_c\right)\max_{m=0,1,2,\ldots,m_c} A_{m}+4 \pi\sqrt{\log \left[\binom{k}{\ell}\binom{n-k}{k'-\ell}\right]}. \label{ratio3}
\end{align}

We now focus on the other case, $m \leq m_c$. Using \eqref{mcok'} it holds,
\begin{align}\label{simple}
\frac{1}{\sqrt{1- \frac{m^2}{(k')^2}}} \leq \frac{1}{\sqrt{1- \frac{1}{4}}} =\frac{2}{\sqrt{3}} \leq 2.
\end{align}

For each $m=0,1,2,\ldots,m_c$ using \eqref{simple} and the simple inequalities $|\{ v,u \in T_{\ell}: \inner{v,u}=m\}| \leq |T_{\ell}|^2=\left[\binom{k'}{\ell} \binom{n-k}{k'-\ell}\right]^2$, $\binom{k'}{\ell} \binom{n-k}{k'-\ell} \leq \binom{n}{k'}$ and $ \binom{n}{k'}\leq n^{k'}$ in that order, we conclude

\begin{align*}
    A_{m} &  \leq 32 \exp\left(  \log \binom{k'}{\ell} \binom{n-k}{k'-\ell} \frac{2m^2}{m^2+(k')^2} \right)\\
    &  \leq 32 \exp\left(  \log \binom{n}{k'} \frac{2m^2}{(k')^2} \right)\\
    &= 32 \exp\left(   \log n \frac{2m^2}{k'} \right)\\
    & \leq 32 \exp\left( 2  \log n \frac{ \left(m_c\right)^2}{k'} \right)\\
\end{align*}which by the definition of $m_c$ and $k'=o(n)$ gives that for some constant $c_0>0,$
$$\max_{m=0,1,2,\ldots,m_c} A_{m}  \leq  \exp\left(  c_0\max\{\frac{ \ell^4}{k'k^2}, \frac{(k')^3}{n^2},1\}\log n\right) .$$ Using \eqref{ratio3} we have for large enough values of $n$,

\begin{align*}
    \frac{\mathbb{E}[Z^2]}{\E{Z}^2}&  \leq 2\left(1+m_c\right)  \exp\left(   c_0\max\{\frac{ \ell^4}{k'k^2}, \frac{(k')^3}{n^2},1\}\log n \right)+4 \pi\sqrt{\log \binom{k}{\ell}\binom{n-k}{k'-\ell}}\\
    &\leq 2\left(1+m_c\right)  \exp\left(   c_0\max\{\frac{ \ell^4}{k'k^2}, \frac{(k')^3}{n^2},1\}\log n \right)+4 \pi\sqrt{\log \binom{n}{k'}}\\
    & \leq 4\left(1+k'\right)  \exp\left(   c_0\max\{\frac{ \ell^4}{k'k^2}, \frac{(k')^3}{n^2},1\}\log n \right)+4 \pi\sqrt{k' \log n}\\
    & \leq c_1k' \log n  \exp\left(   c_0\max\{\frac{ \ell^4}{k'k^2}, \frac{(k')^3}{n^2},1\}\log n \right)
    \end{align*}for some constant $c_1>0.$  We conclude that for some universal constant $C_0>2$,
    \begin{align}
    \frac{\mathbb{E}[Z^2]}{\E{Z}^2}&  \leq    \exp\left(  C_0  \max\left\{\frac{ \ell^4}{k'k^2}, \frac{(k')^3}{n^2},1\right\}\log n \right).
    \end{align}
Now using Paley-Zigmund inequality we conclude that for any $\ell \leq c \min\{k',k\}$ from \eqref{def:psiK}
\begin{align} \mathbb{P}\left(\psi_{k'}(\ell)\geq  2k'\sqrt{\frac{1}{n}\log \left[ \binom{k}{\ell}\binom{n-k}{k'-\ell}\right]}\right) & = \bP \left( Z \geq 1\right) \\
& \geq \frac{\E{Z}^2}{\mathbb{E}[Z^2]} \\
& \geq  \exp\left(  -C_0  \max\left\{\frac{ \ell^4}{k'k^2}, \frac{(k')^3}{n^2},1\right\}\log n \right). \label{PZ}
\end{align}

Since for any $v \in \{0,1\}^n, \|v\|_0=k$, the variance of $v^{\top}Wv$ is $ \frac{2(k')^2}{n}$. Hence, the Borell-TIS inequality on the concentration of the maximum of a Gaussian process implies for any $t>0,$
\begin{align}\label{Borell}
    \bP \left( |\psi_{k'}(\ell) -\E{\psi_{k'}(\ell)}| \geq t\right) \leq \exp\left(-\frac{t^2 n}{4(k')^2} \right).
\end{align}We now choose \begin{align}\label{tstar}
   t^*:= \sqrt{8C_0 \max\left\{\frac{ k'\ell^4}{nk^2}, \frac{(k')^5}{n^3},\frac{(k')^2}{n}\right\}\log n} .
\end{align}

Hence by \eqref{Borell} we have 
\begin{align} \label{Conc0}
    \bP \left( |\psi_{k'}(\ell) -\E{\psi_{k'}(\ell)}| \geq t^*\right) &\leq \exp\left(- \frac{n}{4(k')^2} 8C_0  \max\{\frac{ k'\ell^4}{nk^2}, \frac{(k')^5}{n^3},\frac{(k')^2}{n}\}\log n\right)\\
    &= \exp\left(- 2 C_0  \max\{\frac{ \ell^4}{k'k^2}, \frac{(k')^3}{n^2},1\}\log n \right)
\end{align}
which since $k' \rightarrow +\infty$ it implies that for large values of $n$,
\begin{align} \label{Conc}
    \bP \left( |\psi -\E{\psi}| \geq t^*\right) \leq  \exp\left(- 2 C_0   \max\left\{\frac{ \ell^4}{k'k^2}, \frac{(k')^3}{n^2},1\right\}\log n\right).
\end{align}

In particular using  \eqref{PZ}, for large values of $n$ it holds \begin{align} \label{FRR}
    \bP \left( |\psi_{k'}(\ell) -\E{\psi_{k'}(\ell)}| \geq t^*\right) \leq \mathbb{P}\left(\psi_{k'}(\ell)\geq  2k'\sqrt{\frac{1}{n}\log \left[ \binom{k}{\ell}\binom{n-k}{k'-\ell}\right]}\right).
    \end{align}  For the validity of \eqref{FRR} to hold, we can conclude that $$ 2k'\sqrt{\frac{1}{n}\log \left[ \binom{k}{\ell}\binom{n-k}{k'-\ell}\right]}  \leq \E{\psi_{k'}(\ell)}+t^*$$as otherwise the reverse strict inequality should hold. Therefore $$ 2k'\sqrt{\frac{1}{n}\log \left[\binom{k}{\ell}\binom{n-k}{k'-\ell}\right]}  -2t^* \leq \E{\psi_{k'}(\ell)}-t^*.$$ We conclude 
    
  \begin{align} \label{FR}
      \mathbb{P}\left(\psi_{k'}(\ell)\geq  2k'\sqrt{\frac{1}{n}\log \left[ \binom{k}{\ell}\binom{n-k}{k'-\ell} \right]}-2t^* \right)& \geq  \bP \left( |\psi_{k'}(\ell) -\E{\psi_{k'}(\ell)}| \leq t^*\right)\\
    &\hspace{-2in}\geq 1-\exp\left(- 2 C_0  \max\{\frac{ \ell^4}{k'k^2}, \frac{(k')^3}{n^2},1\}\log n\right)\\
    &\hspace{-2in}=1-\frac{1}{n^3}.
    \end{align}
    where in the second inequality we used \eqref{Conc} and for the last inequality we used the assumption $C_0>2$. Since by \eqref{tstar}, $t^*=\Theta\left( \sqrt{\frac{k'\log n}{n} \max\{\frac{ \ell^4}{k^2}, \frac{(k')^4}{n^2},k'\}}  \right)$ this shows that \eqref{psiK} for sufficiently large values of $n$, holds with probability at least $1-\frac{1}{n^3}$.
    The final argument follows from a union bound over the values of $\ell=0,1,2,\ldots,c \min\{k,k'\}$ and the fact that $k,k' \leq n.$ The proof is complete.
\end{proof}

\section{Proof of Theorem \ref{thm:FEWLow} - Free Energy Wells at Low Temperature}\label{sec:FEWLow}

We make use of the following simple lemma.

\begin{lemma}\label{lem:binom}
If $a,b \in \mathbb{N}$ with $b=o(a)$ and $b \rightarrow +\infty$ then $\log \frac{a!}{b!}=\Theta\left( a\log a\right). $
\end{lemma}

\begin{proof}
From Stirling approximation we have $ \log a!=\Theta\left( a \log a\right)$. By elementary arguments $\log b!=O(b \log b).$ The result follows because $b=o(a).$
\end{proof}

We now start the proof of the theorem.
\begin{proof}[Proof of Theorem \ref{thm:FEWLow}]
For the $\ell_c$ defined in Theorem \ref{thm:FM2}, since $\lambda=o(1) $ and $\ell_c=o\left( \frac{k'}{n}\right)$ we have $\ell_c=\Theta\left(\frac{k}{\lambda}\sqrt{\frac{k'}{n}\log \frac{n}{k'}}\right)$. For this reason, we choose $c_0$ small enough so that $\ell <\frac{1}{4} \ell_c$ and futhermore we have $\ell=\Theta\left(\ell_c\right).$   The asymptotics of $\ell$ follow from that of $\ell_c$ in Theorem \ref{thm:FM2}.

Given Proposition \ref{prop:FEW} and the standard inequality $\binom{n}{k'} \leq k' \log \frac{ne}{k'}$, to show our result, it suffices to show that 
$$\min_{m=\floor{\frac{k'k}{n}},\floor{\frac{k'k}{n}}+1,\ldots,\ell} \phi_{k'}(m)-\min_{m=\ell,\ell+1,\ldots,2\ell} \phi_{k'}(m)=\Theta \left(\frac{k'}{\beta_{\mathrm{Bayes}}}\log \frac{n}{k'}\right)$$with high probability as $n \rightarrow +\infty,$ or equivalently, given the Definition of $\beta_{\mathrm{Bayes}}$, that 
$$\min_{m=\floor{\frac{k'k}{n}},\floor{\frac{k'k}{n}}+1,\ldots,\ell} \phi_{k'}(m)-\min_{m=\ell,\ell+1,\ldots,2\ell} \phi_{k'}(m)=\Theta \left( \frac{k'k}{\lambda n}\log \frac{n}{k'}\right),$$with high probability as $n \rightarrow +\infty.$

Now based on Corollary \ref{cor:approx} it suffices to show
$$\min_{m=\floor{\frac{k'k}{n}},\floor{\frac{k'k}{n}}+1,\ldots,\ell} \Gamma_{k'}(m)-\min_{m=\ell,\ell+1,\ldots,2\ell} \Gamma_{k'}(m)=\Theta \left( \frac{k'k}{\lambda n}\log \frac{n}{k'}\right).$$

Since $\ell < \frac{1}{2}\ell_c$ based on parts (a) and (b) of Theorem \ref{thm:FM2} we know that $\Gamma_{k'}(m)$ is increasing in the regime of $m=0,1,2,\ldots,2\ell$ and therefore 

\begin{align}\label{monoton}\min_{m=\floor{\frac{k'k}{n}},\floor{\frac{k'k}{n}}+1,\ldots,\ell} \Gamma_{k'}(m)-\min_{m=\ell,\ell+1,\ldots,2\ell} \Gamma_{k'}(m)=\Gamma_{k'}(\floor{\frac{k'k}{n}})-\Gamma_{k'}(\ell).\end{align} 

Furthermore based on parts (a) and (b) of Theorem \ref{thm:FM2} we have for each $\ell' \in [0,\ell] \cap \mathbb{N}$, that it holds $$\Gamma_{k'}\left(\ell'+1\right)-\Gamma_{k'}\left(\ell'\right)=\Theta\left(\sqrt{\frac{ k'}{2n  \log \left(\frac{n}{k'}\right)}} \log \left(\frac{\left(\ell+1\right) n}{k' k }\right)\right).$$

Therefore from telescopic summation, we have
\begin{align}\label{telesc}\Gamma_{k'}(\floor{kk'/n})-\Gamma_{k'}(\ell)= \sqrt{\frac{k'}{n \log \frac{n}{k'}}} \log \left( \prod_{m=\floor{kk'/n}}^{\ell} \frac{\left(m+1\right)n}{k'k}\right).\end{align}

Now by Lemma \ref{lem:binom}, since $\ell=\omega\left(\floor{kk'/n}\right)$,
\begin{align}\label{assympt}
\log \left( \prod_{m=\floor{kk'/n}}^{\ell} \frac{\left(m+1\right)n}{k'k}\right)&=\log \left(\frac{ \left(\ell+1\right)! }{\left(\floor{k'k/n}+1\right)!}\right)+\left(\ell-\floor{k'k}{n}\right)\log \frac{n}{k'k}\\
&=\Theta\left(\ell\log \left(\frac{\left(\ell+1\right) n}{k' k }\right)  \right).\end{align}

Combining \eqref{monoton}, \eqref{telesc} and \eqref{assympt} we conclude 
\begin{align}\label{final1}
\min_{m=\floor{\frac{k'k}{n}},\floor{\frac{k'k}{n}}+1,\ldots,\ell} \Gamma_{k'}(m)-\min_{m=\ell,\ell+1,\ldots,2\ell} \Gamma_{k'}(m)=\Theta\left(\sqrt{\frac{k'}{n \log \frac{n}{k'}}} \ell\log \left(\frac{\left(\ell+1\right) n}{k' k }\right)  \right)
\end{align}
Now since $\ell=\Theta\left( \frac{k}{\lambda}\sqrt{\frac{k'}{n} \log \frac{n}{k'}} \right)$ and based on our assumption $\lambda=o(1)$ we conclude 
\begin{align}\label{final2}
\min_{m=\floor{\frac{k'k}{n}},\floor{\frac{k'k}{n}}+1,\ldots,\ell} \Gamma_{k'}(m)-\min_{m=\ell,\ell+1,\ldots,2\ell} \Gamma_{k'}(m)=\Theta\left(\frac{k'k}{n\lambda }\log \left(\frac{n}{k'}\right)\right)
\end{align}The proof is complete.
\end{proof}

\section{Proof of Theorem \ref{thm:OGPLow} - OGP}\label{sec:OGP}

We establish the following theorem, which implies as a direct corollary that $\Phi_{k'}$ exhibits the $k'$-OGP with $\zeta_{2,n}-\zeta_{1,n}=\omega\left(\sqrt{k'} \right),$ with high probability as $n \rightarrow +\infty,$ directly from \eqref{gapOGP} and Proposition \ref{prop:OGP}. 

\begin{theorem}\label{thm:OGP}Let $\lambda>0$ and $k,k',n \in \mathbb{N}$ with $k',k \leq n$ and $k,k',n \rightarrow + \infty$ and $k',k=o\left(n\right).$ Suppose Assumptions \ref{assum:LambdaMain} and \ref{assum:k'Main} hold. Then there exist constants $C_0>1,D_0>0$ and $\ell=\ell_n \in  \mathbb{N}$ such that

\begin{itemize}

    \item $\ell=o\left(\floor{kk'/n} \right), \ell=o\left( \min\{k',k\}\right),$ $\ell=O\left( \frac{k}{\lambda}\sqrt{ \frac{k'}{n} \log \frac{n}{k'}}\right)$ and $\ell=\omega\left(\sqrt{k'}\right)$, and
    
    \item with high probability as $n \rightarrow +\infty$,
 \begin{align} \label{gapOGP} \min_{\ell' \in (\frac{\ell}{3},\frac{\ell}{2})}\phi_{k'}\left(\ell'\right)  \geq \max\{ \phi_{k'}\left(\floor{k'k/n} \right),\phi_{k'}\left( C_0\ell\right)\}+ \mathrm{Gap}_n,
 \end{align}for 

 \begin{align} \label{gapOGPquant}
    \mathrm{Gap}_n:= D_0 \frac{k'k}{\lambda n}.
    \end{align}
\end{itemize}
In particular $\Phi_{k'}$ exhibits the $k'$-OGP with $\zeta_{2,n}-\zeta_{1,n}=\Theta\left(\ell\right)=\omega\left(\sqrt{k'} \right)$, with high probability as $n \rightarrow +\infty.$
\end{theorem}

\begin{proof}[Proof of Theorem \ref{thm:OGP}] 

Using the notation of Theorem \ref{thm:FM2} we choose $\ell=\ell_c$ where $\ell_c$ is defined in Theorem \ref{thm:FM2}. Both parts follow easily from Theorem \ref{thm:FM2}.

For the first part of Theorem \ref{thm:OGP}, the $\ell=o\left(\floor{kk'/n} \right), \ell=o\left( \min\{k',k\}\right), \ell=O\left( \frac{k}{\lambda}\sqrt{ \frac{k'}{n} \log \frac{n}{k'}}\right)$ follow from the first part of Theorem \ref{thm:FM2}. For the last asymptotic equality, notice that using $\lambda=o\left(\min\left\{1,\frac{k}{\sqrt{n}\log n}\right\}\right)$,
\begin{align*}
    \ell&=\frac{1}{2\lambda}k\sqrt{\frac{k'}{n  \log \left(\frac{n}{k'}\right)}} \log \left( \frac{1}{2\lambda}\sqrt{\frac{n}{k'  \log \left(\frac{n}{k'}\right)}}  \right)\\
    &=\Omega\left(\sqrt{k'} \log n \sqrt{\frac{1}{  \log \left(\frac{n}{k'}\right)}} \log \left( \frac{n}{k'}  \right)\right)\\
    &= \omega\left(\sqrt{k'}\right).
\end{align*}

For the second part using Corollary \ref{cor:approx} it suffices to show for the curve $\Gamma_{k'},$
\begin{align} \label{gapOGP2} \min_{\ell' \in (\frac{\ell}{3},\frac{\ell}{2})}\Gamma_{k'}\left(\ell'\right)  \geq \max\{ \Gamma_{k'}\left(\floor{k'k/n} \right),\Gamma_{k'}\left( C_0\ell\right)\}+ \Omega\left(\frac{k'k}{\lambda n}\right).
 \end{align}Notice that since $\ell=\ell_c$ this follows from part (d) of Theorem \ref{thm:FM2} by choosing $\delta_0=\frac{1}{3},\delta_1=\frac{1}{2}$ and $C_0>10.$
\end{proof}

\section{Proof of High Temperature Results}\label{sec:pf-high-temp}

Here we prove our main lower bound on depth (Theorem~\ref{thm:ent-depth} and Corollary~\ref{cor:hightemp-final}) in the high-temperature regime. The proof follows an argument used by \cite{alg-tensor-pca} to show the existence of free energy wells in tensor PCA.

\begin{proof}[Proof of Theorem~\ref{thm:ent-depth}]
Recall $H(v) = -v^\top Y v$ where $Y = \frac{\lambda}{k} xx^\top + W$ with $W \sim \mathrm{GOE}(n)$. Here $x,v \in \{0,1\}^n$ with $\|x\|_0 = k$ and $\|v\|_0 = k'$. For $S \subseteq \{0,1\}^n$ let
\[ Z_\beta(S) = \sum_{v \in S} e^{-\beta H(v)}. \]
Note that
\begin{equation}\label{eq:D1}
D_{\beta,\ell} = \log \mu_\beta(A) - \log \mu_\beta(B) = \log\left(\frac{Z_\beta(A)}{Z_\beta}\right) - \log\left(\frac{Z_\beta(B)}{Z_\beta}\right) = \log Z_\beta(A) - \log Z_\beta(B).
\end{equation}

The first step will be to reduce to the case of the pure noise Hamiltonian $\tilde H(v) = -v^\top W v$. Define $\tilde Z_\beta(S)$ accordingly:
\[ \tilde Z_\beta(S) = \sum_{v \in S} e^{-\beta \tilde H(V)}. \]
Since $H(v) = -\frac{\lambda}{k} \langle v,x \rangle^2 + \tilde H(v)$ we have
\begin{align}
\log Z_\beta(A) - \log Z_\beta(B) &= \log \sum_{v \in A} \exp\left(\frac{\beta\lambda}{k} \langle v,x \rangle^2 - \beta \tilde H(v)\right) - \log \sum_{v \in B} \exp\left(\frac{\beta\lambda}{k} \langle v,x \rangle^2 - \beta \tilde H(v)\right) \nonumber\\
&\ge \log \sum_{v \in A} \exp\left( - \beta \tilde H(v)\right) - \log \sum_{v \in B} \exp\left(\frac{\beta\lambda}{k} (2\ell)^2 - \beta \tilde H(v)\right) \nonumber\\
&= -\frac{4\beta\lambda \ell^2}{k} + \log \tilde{Z}_\beta(A) - \log \tilde{Z}_\beta(B).
\label{eq:D2}
\end{align}

Also,
\begin{equation}\label{eq:D3}
\frac{\tilde Z_\beta(A)}{\tilde Z_\beta(B)} = \frac{\tilde \mu_\beta(A)}{\tilde \mu_\beta(B)} \ge \frac{\tilde \mu_\beta(A)}{\tilde \mu_\beta(A^c)} = \frac{1 - \tilde \mu_\beta(A^c)}{\tilde \mu_\beta(A^c)}
\end{equation}
where $A^c = \{v \in \{0,1\}^n \,:\, \|v\|_0 = k'\} \setminus A$ and $\tilde \mu_\beta$ is the pure noise Gibbs measure $\tilde \mu_\beta(v) \propto e^{-\beta \tilde H(v)}$.

We next apply a symmetry argument to upper bound $\tilde\mu_\beta(A^c)$ in terms of $|A^c|$. Since the distribution of $W$ is spherically symmetric, for any fixed signal $x$ we have
\[ \Ex_W [\tilde\mu_\beta(A^c)] = \frac{|A^c|}{{n \choose k'}} \]
and so by Markov's inequalty, with probability (over $W$) at least $1-\gamma$,
\begin{equation}\label{eq:D4}
\tilde \mu_\beta(A^c) \le \frac{|A^c|}{{n \choose k'}\gamma}.
\end{equation}
Conditioned on this event, and provided $\frac{|A^c|}{{n \choose k'}\gamma} \le \frac{1}{2}$ (which will be ensured by our eventual choice of $\gamma$), we can combine \eqref{eq:D1},\eqref{eq:D2},\eqref{eq:D3},\eqref{eq:D4} to obtain
\begin{equation}\label{eq:D-bound} 
D_{\beta,\ell} \ge -\frac{4 \beta \lambda \ell^2}{k} + \log\left(\frac{1-\tilde\mu_\beta(A^c)}{\tilde\mu_\beta(A^c)}\right) \ge -\frac{4 \beta \lambda \ell^2}{k} + \log\left(\frac{{n \choose k'} \gamma}{2 |A^c|}\right).
\end{equation}

It remains to upper bound $|A^c|$. We have
\[ \frac{|A^c|}{{n \choose k'}} = \sum_{t \in \mathbb{Z} \,:\, \ell \le t \le \min\{k,k'\}} \alpha_t \]
where
\[ \alpha_t := \frac{{k \choose t}{n-k \choose k'-t}}{{n \choose k'}} \le \left(\frac{ke}{t}\right)^t \frac{{n \choose k'-t}}{{n \choose k'}} \]
using the fact that ${k \choose t} \le \left(\frac{ke}{t}\right)^t$ for all $1 \le t \le k$.

We can bound this expression with the help of the following technical lemma, whose proof is deferred to Appendix~\ref{app:proof-binom-ratio}.
\begin{lemma}\label{lem:binom-ratio}
Fix a constant $\delta > 0$. For all $n$ exceeding some constant $n_0 = n_0(\delta)$, for all integers $k'$ and $t$ satisfying $0 \le t \le k' \le n^{1-\delta}$ and $k' \ge 1$,
\[ \log \frac{{n \choose k'-t}}{{n \choose k'}} \le -t (1-\delta) \log\left(\frac{n}{k'}\right). \]
\end{lemma}

Using Lemma~\ref{lem:binom-ratio},
\[ \alpha_t \le \left(\frac{ke}{t}\right)^t \left(\frac{k'}{n}\right)^{t(1-\delta)} = \left(\frac{ke}{t} \left(\frac{k'}{n}\right)^{1-\delta}\right)^t. \]
Provided $\ell \ge 2e k (k'/n)^{1-\delta}$ so that $\alpha_t \le 2^{-t}$ for all $t \ge \ell$, we now have
\begin{equation}\label{eq:A-small}
\frac{|A^c|}{{n \choose k'}} \le 2^{1-\ell}.
\end{equation}

Plugging this back into~\eqref{eq:D-bound}, we have now shown that with probability $1-\gamma$, the depth is
\[ D_{\beta,\ell} \ge -\frac{4 \beta \lambda \ell^2}{k} + (\ell-2) \log 2 + \log \gamma \]
provided $\gamma \ge 2^{-(\ell-2)}$ and $\ell \ge 2e k (k'/n)^{1-\delta}$. Choosing $\gamma = 2^{-(\ell-2)/2}$ completes the proof.
\end{proof}

We now turn to proving Corollary~\ref{cor:hightemp-final}. We first prove a more general statement in the setting where $\ell$ is constrained to some interval $L_1 \le \ell \le L_2$. Corollary~\ref{cor:hightemp-final} will then follow by specializing to our range of informative $\ell$ values.

\begin{corollary}\label{cor:ent-depth}
Fix a constant $\delta > 0$ and suppose $n \ge n_0(\delta)$ and $k' \le n^{1-\delta}$. Let
\[ \max\{1/2, 2ek(k'/n)^{1-\delta}\} \le L_1 \le L_2 \le \frac{1}{2}\min\{k,k'\}. \]
Let $B_1 = \frac{\log 2}{16}\frac{k}{\lambda L_1}$. For any $0 \le \beta \le B_1$ there exists $L_1 \le \ell \le L_2$ such that with probability at least $1-2^{-(\ell-2)/2}$,
\[ D_{\beta,\ell} \ge \frac{\log 2}{4} \min\left\{\frac{\log 2}{16} \frac{k}{\beta \lambda}, L_2\right\} - \log 2. \]
\end{corollary}

\begin{proof}
Let $B_2 = \frac{\log 2}{16} \frac{k}{\lambda L_2}$. First consider the case $B_2 \le \beta \le B_1$. In this case, set $\ell = \frac{\log 2}{16} \frac{k}{\beta \lambda}$ and note that this satisfies $L_1 \le \ell \le L_2$. Theorem~\ref{thm:ent-depth} gives $D_{\beta,\ell} \ge \frac{\log^2 2}{64} \frac{k}{\beta \lambda} - \log 2$. Now consider the case $0 \le \beta \le B_2$. In this case, set $\ell = L_2$, which means $\ell \le \frac{\log 2}{16} \frac{k}{\beta\lambda}$. Theorem~\ref{thm:ent-depth} gives $D_{\beta,\ell} \ge -\frac{4\beta\lambda}{k} \cdot \frac{\log 2}{16} \frac{k}{\beta\lambda} \cdot L_2 + \frac{\log 2}{2}L_2 - \log 2 = \frac{\log 2}{4} L_2 - \log 2$.
\end{proof}

\begin{proof}[Proof of Corollary~\ref{cor:hightemp-final}]
Apply Corollary~\ref{cor:ent-depth} with $L_1 = \max\left\{n^\delta,2ek(k'/n)^{1-\delta}\right\}$ and $L_2 = \frac{k}{2\lambda}\sqrt{\frac{k'}{n}}$. To ensure $L_1 \le L_2$ we need
\[ \lambda \le \min\left\{n^{-\delta} \frac{k}{2} \sqrt{\frac{k'}{n}}, \frac{1}{4e} \left(\frac{n}{k'}\right)^{1/2-\delta}\right\}. \]
We have
\[ B_1 = \frac{\log 2}{16} \frac{k}{\lambda} \min\{n^{-\delta},(2ek)^{-1}(n/k')^{1-\delta}\}
\ge \frac{\log 2}{16} n^{-\delta} \min\left\{\frac{k}{\lambda},\frac{n}{2e\lambda k'}\right\}. \]
\end{proof}

\appendix

\section{Additional Proofs}\label{app:technical}

\subsection{Proof of Remark~\ref{rem:tight}}
\label{app:pf-random-walk}

In this section we prove the following simple fact claimed in Remark~\ref{rem:tight}.

\begin{proposition}
Let $\mathcal{G}$ be the graph with vertex set $\{v \in \{0,1\}^n \,:\, \|v\|_0 = k'\}$, with an edge between each pair of vertices whose associated vectors differ in exactly 2 coordinates. Let $X_0, X_1, X_2, \ldots$ be the Markov chain on the vertices of $\mathcal{G}$ where $X_0$ is a uniformly random initialization and $X_{i+1}$ is a uniformly random neighbor of $X_i$. Fix a vertex $x$ and let $\tau$ be the hitting time $\tau = \inf\{t \in \mathbb{N} \,:\, X_\tau = x\}$. Then for any $t \ge 0$ we have $\Pr\{\tau \ge t\} \le k' n^{2k'}/t$.
\end{proposition}
\begin{proof}
Let $y$ be the current state. There exists a path of length $\ell \le k'$ from $y$ to $x$. Since $\mathcal{G}$ is a regular graph of degree $d := k'(n-k') \le n^2$, the probability of following this path over the next $\ell$ steps (and thus reaching $x$) is at least $d^{-\ell} \ge d^{-k'}$. Let $N$ be the number of such trials (each consisting of at most $k'$ steps) before $x$ is reached. We have $\tau \le k' N$ and $\mathbb{E}[N] \le d^{k'} \le n^{2k'}$, so $\mathbb{E}[\tau] \le k' n^{2k'}$. The result follows by Markov's inequality.
\end{proof}

\subsection{Proof of Proposition \ref{prop:mcmc}}\label{app:mcmc}

\begin{proof}[Proof of Proposition \ref{prop:mcmc}]
Suppose $X_0 \sim \mu_\beta$ (not conditioning on $A$ yet) and $X_1,X_2,\ldots$ are drawn according to the Markov chain. Each $X_i$ is distributed according to $\mu_\beta$ (although they are not independent). Using Bayes' rule and a union bound,
\begin{align*}
\Pr\{\tau_\beta \le t\} &= \Pr\{\exists i \in \{1,\ldots,t\} : X_i \in B \,|\, X_0 \in A\} \\
&= \frac{\Pr\{X_0 \in A \text{ and } \exists i \in \{1,\ldots,t\} : X_i \in B\}}{\Pr\{X_0 \in A\}} \\
&\le \frac{\Pr\{\exists i \in \{1,\ldots,t\} : X_i \in B\}}{\Pr\{X_0 \in A\}} \\
&\le \frac{\sum_{i=1}^t \Pr\{X_i \in B\}}{\Pr\{X_0 \in A\}} \\
&= \frac{t \mu_\beta(B)}{\mu_\beta(A)} \\
&= t \exp(-D_{\beta,\ell}).
\end{align*}
\end{proof}

\subsection{Proof of Theorem~\ref{thm:main-informal}}\label{app:main-pf}

\begin{proof}[Proof of Theorem~\ref{thm:main-informal}]
First consider the low-temperature regime $\beta \gg \frac{\lambda n}{k}$, in which case our low-temperature bound gives a free energy well of depth $\Omega(k')$ provided that the additional assumption $k' \ll \sqrt[3]{\frac{k^2 n}{\lambda^2}}$ is satisfied. In the case (i) $k' \le k$, the additional assumption follows from the assumption $\lambda \ll 1$. In the case (ii) $\lambda \ll (k/n)^{1/4}$, the additional assumption follows from the ``informative'' $k'$ assumption $k' \le \frac{n \lambda^2}{\log n}$. Thus, in the low-temperature regime $\beta \gg \frac{\lambda n}{k}$ we have a free energy well of depth $\Omega(k')$. Using the ``informative'' $k'$ assumption $k' \ge \frac{k^2 \log n}{\lambda^2 n}$, this depth is $\Omega\left(\frac{k^2}{\lambda^2 n}\right)$ as desired. Here, $\beta \gg \frac{\lambda n}{k}$ specifically means $\beta \ge \frac{\lambda n}{k} \cdot \mathrm{polylog}(n)$; see Appendix~\ref{app:increase-lambda}.

Now consider the remaining temperature values $\beta \le \frac{\lambda n}{k} \cdot \mathrm{polylog}(n)$. In order for our high-temperature bound to cover this entire regime, we need $\frac{\lambda n}{k} \ll \min\{\frac{k}{\lambda},\frac{n}{\lambda k'}\}$. One bound $\frac{\lambda n}{k} \ll \frac{k}{\lambda}$ follows from the assumption $\lambda \ll k/\sqrt{n}$. The other bound $\frac{\lambda n}{k} \ll \frac{n}{\lambda k'}$, i.e., $\lambda^2 \ll k/k'$, can be shown as follows: if (i) $k' \le k$ then this follows from $\lambda \ll 1$, and if (ii) $\lambda \ll (k/n)^{1/4}$ then this follows from the ``informative'' $k'$ assumption $k' \le \frac{n \lambda^2}{\log n}$. Thus, the high-temperature bound applies for all $\beta \le \frac{\lambda n}{k} \cdot \mathrm{polylog}(n)$ and gives a free energy well of depth $\Omega\left(\min\left\{\frac{k}{\beta\lambda},\frac{k}{\lambda}\sqrt{\frac{k'}{n}}\right\}\right)$. Using the bounds $\beta \le \frac{\lambda n}{k} \cdot \mathrm{polylog}(n)$ and $k' \ge \frac{k^2 \log n}{\lambda^2 n}$, this depth is $\tilde\Omega\left(\frac{k^2}{\lambda^2 n}\right)$ as desired.
\end{proof}

\subsection{Extending the Range of $k'$ Values}
\label{app:increase-lambda}

As discussed in the main text, the assumption~\eqref{k'Assum1} on $k'$ used for our low-temperature results does not quite cover all the informative $k'$ values as defined in~\eqref{eq:k-good}. Here we explain a simple argument that allows us to obtain essentially the same (up to log factors) lower bound on $D_{\beta,\ell}$ as Theorem~\ref{thm:FEWLow} for the entire range of reasonable $k'$ values. Suppose we are in a setting where~\eqref{eq:k-good} is satisfied but not~\eqref{k'Assum1}. Choose a slightly larger $\lambda$ value $\tilde\lambda = \lambda \cdot \mathrm{polylog}(n)$ so that $\tilde\lambda$ satisfies~\eqref{k'Assum1}. Apply Theorem~\ref{thm:FEWLow} with $\tilde\lambda$ in place of $\lambda$ (and all other parameters unchanged); this requires the original $\lambda$ to satisfy a slight strengthening (by log factors) of the assumptions~\eqref{assum:lambda} and \eqref{addit_assum}. The result is
\[ D_{\beta,\ell}(\tilde\lambda) \ge d_1 \left(\frac{\beta}{\beta_{\mathrm{Bayes}}(\tilde\lambda)}-d_2\right)k'\log \left( \frac{n}{k'}\right) \]
where $D_{\beta,\ell}(\tilde\lambda)$ and $\beta_{\mathrm{Bayes}}(\tilde\lambda) = \frac{\tilde\lambda n}{2k}$ denote the corresponding quantities for $\tilde\lambda$. Here the value of $\ell$ satisfies $\ell=\Theta\left(\frac{k}{\tilde\lambda}\sqrt{\frac{k'}{n}\log \frac{n}{k'}}\right)$, which is informative (satisfying~\eqref{eq:ell-good}) for the original $\lambda$ value provided $\min\{1,kk'/n\} \ll \frac{k}{\lambda}\sqrt{\frac{k'}{n}}$ (which follows from $\lambda$ being in the hard regime: $\sqrt{k/n} \ll \ell \ll \min\{1,k/\sqrt{n}\}$). Lemma~\ref{lem:lambda-monotone} below shows that $D_{\beta,\ell}$ is monotone in $\lambda$ and so $D_{\beta,\ell}(\lambda) \ge D_{\beta,\ell}(\tilde\lambda)$. Thus we get a lower bound of $\Omega(k' \log(n/k'))$ on $D_{\beta,\ell}(\lambda)$ for all $\beta \ge C \beta_{\mathrm{Bayes}}(\tilde\lambda)$ where $C$ is a universal constant, i.e., for all $\beta \ge \beta_\mathrm{Bayes}(\lambda) \cdot \mathrm{polylog}(n)$. Thus we conclude the low-temperature lower bound (2) as stated in Section~\ref{sec:pf-tech}.

\begin{lemma}\label{lem:lambda-monotone}
If $\tilde\lambda \ge \lambda$ then $D_{\beta,\ell}(\lambda) \ge D_{\beta,\ell}(\tilde\lambda)$.
\end{lemma}
\begin{proof}
Using the definition of $D_{\beta,\ell}$ it suffices to show that
\begin{equation}\label{eq:goal}
    \frac{ \sum_{v \in A} e^{\beta \lambda \inner{x,v}^2+v^\top Wv}}{\sum_{v \in B} e^{\beta \lambda \inner{x,v}^2+v^\top Wv}} \geq \frac{ \sum_{v \in A} e^{\beta \tilde \lambda \inner{x,v}^2+v^\top Wv}}{\sum_{v \in B} e^{\beta \tilde \lambda \inner{x,v}^2+v^\top Wv}}.
\end{equation}

Notice though that for any $v \in A,$ $$\lambda \inner{v,x}^2 =\tilde \lambda \inner{v,x}^2+(\lambda-\tilde \lambda )\inner{v,x}^2 \geq  \tilde \lambda \inner{v,x}^2+(\lambda-\tilde \lambda )\ell $$ and for any $v \in B,$
$$\lambda \inner{v,x}^2 =\tilde \lambda \inner{v,x}^2+(\lambda-\tilde \lambda )\inner{v,x}^2 \leq  \tilde \lambda \inner{v,x}^2+(\lambda-\tilde \lambda )\ell. $$ 

Hence,
\begin{equation*}
    \frac{ \sum_{v \in A} e^{\beta \lambda \inner{x,v}^2+v^\top Wv}}{\sum_{v \in B} e^{\beta \lambda \inner{x,v}^2+v^\top Wv}} \geq \frac{ \sum_{v \in A} e^{\beta \tilde \lambda \inner{x,v}^2+(\lambda-\tilde \lambda )\ell+v^\top Wv}}{\sum_{v \in B} e^{\beta \tilde \lambda \inner{x,v}^2+(\lambda-\tilde \lambda )\ell+v^\top Wv}}=\frac{ \sum_{v \in A} e^{\beta \tilde \lambda \inner{x,v}^2+v^\top Wv}}{\sum_{v \in B} e^{\beta \tilde \lambda \inner{x,v}^2+v^\top Wv}}.
\end{equation*}
This completes the proof.
\end{proof}

\subsection{Auxiliary Lemmas}

\begin{lemma}\label{lem:calc}
Suppose $a_n,b_n, n\in \mathbb{N}$ are two positive-valued sequences with $\lim_n b_n = + \infty$. Then $$\limsup_n \frac{a_n \log a_n}{b_n}=\limsup_n \frac{a_n}{b_n  \log b_n}$$and $$\liminf_n \frac{a_n \log a_n}{b_n}=\liminf_n \frac{a_n}{b_n  \log b_n}.$$
\end{lemma}
\begin{proof}
We prove the first equality, as the second equality follows by similar considerations.

By passing to a subsequence we can assume that  $\frac{a_n}{b_n  \log b_n}$ is converging to some nonnegative value $C \in [0,+\infty].$ If $C=+\infty$ the result follows by \cite[Proposition 4]{gamarnik2019landscape}. If $C<+ \infty$ then for any $\epsilon>0$ and sufficiently large $n$,
 $$(C-\epsilon)b_n \log b_n < a_n<(C+\epsilon)b_n \log b_n. $$
 In particular it holds 
 $$(C-\epsilon) \frac{\log ((C-\epsilon)b_n)}{\log b_n }\frac{a_n \log a_n}{b_n} \leq (C+\epsilon) \frac{\log ((C+\epsilon)b_n)}{\log b_n }.$$The fact that $C<+\infty$ and $\lim_n b_n=+\infty$ completes the proof.
\end{proof}

\begin{lemma}\label{comb:lemma}
Under the parameter assumptions of Theorem \ref{FM} there exists some sufficiently small constant $c>0$ such that the following holds. 
\begin{itemize}
    \item[(a)] For large values of $n$, $\max\{\frac{32 \ell^2}{k}, \frac{16 (k')^2}{n}\} \leq \frac{k'}{2}.$
    \item[(b)] Suppose  $k'-1 \geq m \geq \max\{\frac{32 \ell^2}{k}, \frac{16 (k')^2}{n}\}.$ Then
\begin{align}\label{ratioBin2}\frac{ |\{ v,u \in T_{\ell}: \inner{v,u}=m+1\}|}{|\{ v,u \in T_{\ell}: \inner{v,u}=m\}|} \leq \frac{1}{2},
\end{align}where $$T_{\ell}:=\{ v \in \{0,1\}^n: \|v\|_0=k', \inner{v,x}=\ell\}.$$
\end{itemize}
\end{lemma}

\begin{proof}
For the first part notice that since $k'=o(n)$ for large values of $n$, $16 \frac{(k')^2}{n} <\frac{k'}{2}.$ Furthermore by choosing $c< \frac{1}{64}$ we have that $\ell \leq c \min\{k',k\}$ implies $$\frac{32 \ell^2}{k}=32 \frac{\ell}{64} \frac{\ell}{k} <32 \frac{k'}{64}=\frac{k'}{2}.$$

We now turn to the second part.
Now notice that for each $m=0,1,2,\ldots,k',$ from elementary counting arguments based on $m_0$, the value of the common intersection size of the supports of all $v,u,x$,
$$|\{ v,u \in T_{\ell}: \inner{v,u}=m\}|=\binom{k}{\ell}\binom{n-k}{k-\ell}\left[\sum_{m_0=0}^{\min\{m,\ell\}} \binom{\ell}{m_0}\binom{k-\ell}{\ell-m_0}\binom{k'-\ell}{m-m_0}\binom{n-k-k'}{k'-m-\ell+m_0}\right].$$

Hence, for all $m,$ \begin{align}\label{ratioBin}\frac{ |\{ v,u \in T_{\ell}: \inner{v,u}=m+1\}|}{|\{ v,u \in T_{\ell}: \inner{v,u}=m\}|}=\frac{\sum_{m_0=0}^{\min\{m+1,\ell\}} \binom{\ell}{m_0}\binom{k-\ell}{\ell-m_0}\binom{k'-\ell}{m+1-m_0}\binom{n-k-k'}{k'-m-1-\ell+m_0}}{\sum_{m_0=0}^{\min\{m,\ell\}} \binom{\ell}{m_0}\binom{k-\ell}{\ell-m_0}\binom{k'-\ell}{m-m_0}\binom{n-k-k'}{k'-m-\ell+m_0}}.\end{align}

Recall the elementary identity that for any $a,b \in \mathbb{N}$ with $b \leq a$ it holds $\binom{a}{b+1}=\frac{a-b}{b+1}\binom{a}{b}.$ Hence, for all $m_0=0,1,2,\ldots,\min\{m,\ell\},$  it holds 
\begin{align*}
  \frac{\binom{\ell}{m_0}\binom{k-\ell}{\ell-m_0}\binom{k'-\ell}{m+1-m_0}\binom{n-k-k'}{k'-m-1-\ell+m_0}}{\binom{\ell}{m_0}\binom{k-\ell}{\ell-m_0}\binom{k'-\ell}{m-m_0}\binom{n-k-k'}{k'-m-\ell+m_0}} &= \frac{\binom{k'-\ell}{m+1-m_0}\binom{n-k-k'}{k'-m-1-\ell+m_0}}{\binom{k'-\ell}{m-m_0}\binom{n-k-k'}{k'-m-\ell+m_0}}\\
  & = \frac{k'-\ell-m+m_0}{m+1-m_0} \frac{k'-m-\ell+m_0}{n-k-2k'+m+1+\ell-m_0}.
\end{align*}We focus now on $m_0=0,1,2,\ldots,\min\{\floor{\frac{m}{2}},\ell\},$ where $m_0 \leq \frac{m}{2}$ holds. Combining $m_0 \leq \frac{m}{2}$ with the last displayed equation and $k,k'=o(n)$ we have that for large values of $n$, it holds
\begin{align*}
\frac{\binom{\ell}{m_0}\binom{k-\ell}{\ell-m_0}\binom{k'-\ell}{m+1-m_0}\binom{n-k-k'}{k'-m-1-\ell+m_0}}{\binom{\ell}{m_0}\binom{k-\ell}{\ell-m_0}\binom{k'-\ell}{m-m_0}\binom{n-k-k'}{k'-m-\ell+m_0}}\leq  \frac{(k')^2}{\frac{m}{2}\frac{n}{2}}=\frac{4(k')^2}{mn}.
\end{align*}Since by assumption $m \geq \frac{32(k')^2}{n}$  for $m_0=0,1,2,\ldots,\min\{\floor{\frac{m}{2}},\ell\},$ it holds 
\begin{align}\label{binom1}
\frac{\binom{\ell}{m_0}\binom{k-\ell}{\ell-m_0}\binom{k'-\ell}{m+1-m_0}\binom{n-k-k'}{k'-m-1-\ell+m_0}}{\binom{\ell}{m_0}\binom{k-\ell}{\ell-m_0}\binom{k'-\ell}{m-m_0}\binom{n-k-k'}{k'-m-\ell+m_0}} \leq \frac{1}{8}.
\end{align}

Notice that if $\floor{\frac{m}{4}}>\ell$ then \eqref{binom1} holds for all $m=0,1,2,\ldots,\min\{m,\ell\}=\ell.$ Now assuming that $\floor{\frac{m}{4}}<\ell$ we consider the regime where $m_0=\floor{\frac{m}{4}}+1,\ldots,\ldots,\min\{m,\ell\}$. In that regime, we first observe that
\begin{align*}
  \frac{\binom{\ell}{m_0+1}\binom{k-\ell}{\ell-m_0-1}\binom{k'-\ell}{m-m_0}\binom{n-k-k'}{k'-m-\ell+m_0}}{\binom{\ell}{m_0}\binom{k-\ell}{\ell-m_0}\binom{k'-\ell}{m-m_0}\binom{n-k-k'}{k'-m-\ell+m_0}}&=\frac{\binom{\ell}{m_0+1}\binom{k-\ell}{\ell-m_0-1}}{\binom{\ell}{m_0}\binom{k-\ell}{\ell-m_0}}\\
  &\leq \frac{(\ell-m_0)^2}{(m_0+1)(k-2\ell+m_0)}.
\end{align*}Now since $m_0 \geq m/4$ there exists a sufficiently small constant $c>0$ such that if $\ell<ck$ then for large values of $n$, it holds 
\begin{align*}
  \frac{\binom{\ell}{m_0+1}\binom{k-\ell}{\ell-m_0-1}\binom{k'-\ell}{m-m_0}\binom{n-k-k'}{k'-m-\ell+m_0}}{\binom{\ell}{m_0}\binom{k-\ell}{\ell-m_0}\binom{k'-\ell}{m-m_0}\binom{n-k-k'}{k'-m-\ell+m_0}} \leq \frac{8\ell^2}{mk}.
\end{align*}
Since by assumption $m \geq \frac{64 \ell^2}{k},$ for large values of $n$, for all $m_0=\floor{\frac{m}{4}}+1,\ldots,\ldots,\min\{m,\ell\}$ it holds 

\begin{align}\label{binom2}
  \frac{\binom{\ell}{m_0+1}\binom{k-\ell}{\ell-m_0-1}\binom{k'-\ell}{m-m_0}\binom{n-k-k'}{k'-m-\ell+m_0}}{\binom{\ell}{m_0}\binom{k-\ell}{\ell-m_0}\binom{k'-\ell}{m-m_0}\binom{n-k-k'}{k'-m-\ell+m_0}} \leq \frac{1}{8}.
\end{align}

Observe that the following quantity $$\sum_{m_0=0}^{m+1} \binom{\ell}{m_0}\binom{k-\ell}{\ell-m_0}\binom{k'-\ell}{m+1-m_0}\binom{n-k-k'}{k'-m-1-\ell+m_0}$$is at most
\begin{align*}
& \sum_{m_0=0}^{\floor{m/2}} \binom{\ell}{m_0}\binom{k-\ell}{\ell-m_0}\binom{k'-\ell}{m+1-m_0}\binom{n-k-k'}{k'-m-1-\ell+m_0}\\
&+\sum_{m_0=\ceil{m/4}+1}^{m+1} \binom{\ell}{m_0}\binom{k-\ell}{\ell-m_0}\binom{k'-\ell}{m+1-m_0}\binom{n-k-k'}{k'-m-1-\ell+m_0}\\
\end{align*}
which equals
\begin{align*}
& \sum_{m_0=0}^{\floor{m/2}} \binom{\ell}{m_0}\binom{k-\ell}{\ell-m_0}\binom{k'-\ell}{m+1-m_0}\binom{n-k-k'}{k'-m-1-\ell+m_0}\\
&+\sum_{m_0=\ceil{m/4}}^{m} \binom{\ell}{m_0+1}\binom{k-\ell}{\ell-m_0-1}\binom{k'-\ell}{m-m_0}\binom{n-k-k'}{k'-m-\ell+m_0},
\end{align*}where the inequality holds as all the terms are nonnegative and the equality  follows by a simple change of variables.

Now we further upper bound the last quantity by using \eqref{binom1} for the first summand and \eqref{binom2} for the second summand to get

\begin{align*}
&\sum_{m_0=0}^{m+1} \binom{\ell}{m_0}\binom{k-\ell}{\ell-m_0}\binom{k'-\ell}{m+1-m_0}\binom{n-k-k'}{k'-m-1-\ell+m_0}\\
&\leq \frac{1}{8}\sum_{m_0=0}^{\floor{m/2}} \binom{\ell}{m_0}\binom{k-\ell}{\ell-m_0}\binom{k'-\ell}{m-m_0}\binom{n-k-k'}{k'-m-\ell+m_0}\\
&+\frac{1}{8}\sum_{m_0=\ceil{m/4}}^{m} \binom{\ell}{m_0}\binom{k-\ell}{\ell-m_0}\binom{k'-\ell}{m-m_0}\binom{n-k-k'}{k'-m-\ell+m_0}\\
\end{align*}which is at most
\begin{align*}
& \left(\frac{1}{8}+\frac{1}{8}\right)\sum_{m_0=0}^{m}\binom{\ell}{m_0}\binom{k-\ell}{\ell-m_0}\binom{k'-\ell}{m-m_0}\binom{n-k-k'}{k'-m-\ell+m_0}\\
&=\frac{1}{4}\sum_{m_0=0}^{m}\binom{\ell}{m_0}\binom{k-\ell}{\ell-m_0}\binom{k'-\ell}{m-m_0}\binom{n-k-k'}{k'-m-\ell+m_0},
\end{align*}where the last inequality holds because all the terms are nonnegative. Using \eqref{ratioBin} completes the proof of the lemma.
\end{proof}

\begin{lemma}\label{lem:Small}
Under Assumptions \ref{assum:LambdaMain} and \ref{assum:k'Main}, the following holds.
If $\ell=\Theta\left( \frac{k}{\lambda}\sqrt{ \frac{k'}{n} \log \frac{n}{k'}}\right)$ then $\ell=o\left( \min\{k',k\}\right)$ and $\ell=\omega\left(\frac{k'k}{n} \right).$
\end{lemma}

\begin{proof}
Define as in Theorem \ref{thm:FM2} $$\ell_c:=\frac{1}{2\lambda}k\sqrt{\frac{k'}{n  \log \left(\frac{n}{k'}\right)}} \log \left( \frac{1}{2\lambda}\sqrt{\frac{n}{k'  \log \left(\frac{n}{k'}\right)}}  \right).$$ Now notice that by Assumption \ref{assum:LambdaMain} it holds for sufficiently large $n$, $$\lambda<1$$ and by Assumption \ref{assum:LambdaMain}, $$\lambda>\sqrt{\frac{k'}{n}}.$$ Hence, $$\log \left( \frac{1}{2\lambda}\sqrt{\frac{n}{k'  \log \left(\frac{n}{k'}\right)}}\right)=\Theta\left(\log \frac{n}{k'} \right)$$ which implies $$\ell_c=\Theta\left( \frac{k}{\lambda}\sqrt{ \frac{k'}{n} \log \frac{n}{k'}}\right)$$ or $\ell=\Theta\left(\ell_c\right).$ The result follows from the first part of Theorem \ref{thm:FM2} on the asymptotic behavior of $\ell_c.$
\end{proof}

\subsection{Proof of Corollary~\ref{cor:approx}}\label{app:approx}

\begin{proof}[Proof of Corollary \ref{cor:approx}]

Based on Theorem \ref{FM} it suffices to show that under our assumptions,
\begin{align}\label{goalLem} \sqrt{\frac{(k')^2}{n \log \left( \binom{k}{\ell} \binom{n-k}{k'-\ell} \right)}}\left(\log n\right)^2+O\left(\sqrt{\frac{k'\log n}{n} \max\left\{\frac{ \ell^4}{k^2}, \frac{(k')^4}{n^2},k' \right\}}\right)=o\left(\frac{k'k}{\lambda n}\right).\end{align} 
Since $\ell=o(k')$ by Lemma \ref{lem:Small}, and $k,k'=o(n)$, for large enough values of $n$ it holds $$\ell \leq \frac{k'}{2}, k'  \leq \frac{n-k}{2}.$$ Therefore combining with the elementary inequality $\binom{a}{b} \geq \left(\frac{a}{b}\right)^b$ it holds $$\binom{k}{\ell} \binom{n-k}{k'-\ell} \geq  \binom{n-k}{\floor{k'/2}}  \geq \left( \frac{n-k}{\floor{k'/2}}\right)^{\floor{k'/2}}.$$ Since $k,k'=o(n)$ it also holds for large values of $n$ $$\frac{n-k}{\floor{k'/2}} \geq e^3$$which allows us to conclude that for large values of $n$, 
$$\log \left[\binom{k}{\ell} \binom{n-k}{k'-\ell} \right] \geq \floor{k'/2} \log \left( \frac{n-k}{\floor{k'/2}}\right) \geq  3\floor{k'/2} \geq k'.$$
Hence to show \eqref{goalLem} it suffices to show
$$\frac{k'k}{\lambda n}=\omega \left(\max\left\{\sqrt{\frac{k'\log n}{n} \max\left\{\frac{ \ell^4}{k^2}, \frac{(k')^4}{n^2},k' \right\}},\sqrt{\frac{k'}{n}}\left(\log n\right)^2\right\}\right).$$

We now compare the four terms. For the first term, we have that $$\frac{k'k}{\lambda n}=\omega\left( \sqrt{\frac{k' \log n}{n}\frac{\ell^4}{k^2}} \right)$$  if and only if $$\ell=o\left(\sqrt{\sqrt{\frac{k'}{n \log n}}\frac{k^2}{\lambda }}\right).$$
We have by construction $$\ell=O\left(\frac{k}{\lambda}\sqrt{\frac{k'}{n} \log n }\right)$$ and therefore the result follows since according to Assumption \ref{assum:k'Main}, $$k'=o \left(\frac{\lambda^2 n}{\log^3 n}\right).$$ For the second term, $$\frac{k'k}{\lambda n} =\omega \left( \sqrt{\frac{k'\log n}{n}\frac{(k')^4}{n^2} }\right)$$ holds if and only if $$(k')^3=o\left( \frac{k^2 n}{\lambda^2 \log n} \right)$$ or 
$$k'=o\left( \left(\frac{k^2 n}{\lambda^2 \log n}\right)^{\frac{1}{3}} \right)$$  which follows from Assumption \ref{assum:k'Main}. For the third term, $$\frac{k'k}{\lambda n} =\omega \left( \sqrt{\frac{k'\log n}{n}k' }\right)$$ holds if and only if $$\lambda=o\left(\frac{k}{\sqrt{n \log n}} \right)$$ which is assumed to be true in Assumption \ref{assum:LambdaMain}. For the fourth term, $$\frac{k'k}{\lambda n} =\omega \left( \sqrt{\frac{k'}{n}}\left(\log n\right)^2 \right)$$ which holds if and only if $$k'=\omega \left( \frac{\lambda^2 n}{k^2} (\log n)^4\right).$$ By assumption \ref{assum:k'Main},  $$k'=\Omega\left(\frac{k^2}{\lambda^2 n}\right)$$ and therefore it suffices to have $$\frac{k^2}{\lambda^2 n}=\omega\left( (\log n)^{2} \right)$$ or $$\lambda=o\left( \frac{k}{\sqrt{n} \log n}\right)$$ which is assumed to be true in Assumption \ref{assum:LambdaMain}.
\end{proof}

\subsection{Proof of Proposition \ref{prop:OGP}}\label{app:OGP}

\begin{proof}[Proof of Proposition \ref{prop:OGP}] We choose $\zeta_{1,n}:=z_1,\zeta_{2,n}:=z_2$ and $r_n>0$ some value with $\max\{ \phi_{k'}\left(\ell_1\right),\phi_{k'}\left(\ell_2\right)\} <r_n\leq \min_{\ell \in (z_{1},z_{2})} \phi_{k'}\left(\ell\right)$.

For the first condition, we choose  $v,w$ binary $k'$-sparse vectors the optimal solutions of $\Phi_{k'}(\ell_1), \Phi_{k'}(\ell_2)$ respectively.
Since $\phi_{k'}(\ell_1) <r_n$ and $\phi_{k'}(\ell_2) <r_n$ it holds$$\max \{ v^{\top}Yv, w^{\top}Yw\} \leq \max \{ \phi_{k'}(\ell_1), \phi_{k'}(\ell_2)\} < r_n.$$
Furthermore, $$\inner{v,x} =\ell_1 \leq z_1=\zeta_{1,n}$$ and  $$\inner{w,x} =\ell_2 \geq z_2=\zeta_{2,n}.$$

For the second condition, as $\min_{\ell \in (z_{1},z_{2})} \phi_{k'}\left(\ell\right)>r_n$  by definition no $v$ binary $k'$-sparse vector with $\inner{v,x} \in (\zeta_{1,n},\zeta_{2,n})=(z_{1},z_{2})$ satisfies $v^{\top}Yv \leq r_n$.
\end{proof}

\subsection{Proof of Proposition \ref{prop:FEW}}\label{app:FEW}

\begin{proof}[Proof of Proposition \ref{prop:FEW}] From Definition \ref{def:FEW},
$$D_{\beta,\ell}=\log \frac{\mu_{\beta}(A)}{\mu_{\beta}(B)}=\log \frac{\sum_{v \in A}e^{-\beta H(v)} }{\sum_{v \in B}e^{-\beta H(v)}}.$$
Now notice that from the definition of overlap and the $\phi$ curve,
$$\min_{v \in A}e^{-\beta H(v)}=\exp\left(\beta \min_{m=\floor{\frac{k'k}{n}},\floor{\frac{k'k}{n}}+1,\ldots,\ell} \phi_{k'}(m)\right)$$
and
$$\min_{v \in B}e^{-\beta H(v)}=\exp\left(\beta \min_{m=\ell,\ell+1,\ldots,2\ell} \phi_{k'}(m)\right).$$
For this reason, to prove our result it suffices to establish the elementary claim that for two sequences of negative numbers $\{a_i\}_{i = 1,2,\ldots,N}$ and $\{b_i\}_{i=1,2,\ldots,M}$,
\begin{align}\label{elem2}\Bigg| \log \frac{ \sum_{i=1}^N a_i}{\sum_{i=1}^M b_i}-\log\left(\frac{\min_{i=1,\ldots,N} a_i}{\min_{i=1,\ldots,M} b_i}\right)\Bigg| \leq \log \left(\max\{N,M\}\right).\end{align} The result then follows by choosing the $a$ sequence corresponding to $e^{-\beta H(v)}, v \in A$, the $b$ sequence corresponding to $e^{-\beta H(v)}, v \in B$ and finally that $\max\{|A|,|B|\} \leq \binom{n}{k'}$ as both $A,B$ are subsets of the $k'$-sparse binary vectors.

To show \eqref{elem2} it suffices to show equivalently for two sequences of positive numbers $\{a_i\}_{i = 1,2,\ldots,N}$ and $\{b_i\}_{i=1,2,\ldots,M}$,
\begin{align}\label{elem}\Bigg| \log \frac{ \sum_{i=1}^N a_i}{\sum_{i=1}^M b_i}-\log\left(\frac{\max_{i=1,\ldots,N} a_i}{\max_{i=1,\ldots,M} b_i}\right)\Bigg| \leq \log \left(\max\{N,M\}\right).\end{align}

Now to prove \eqref{elem} it suffices to show
$$\frac{\max_{i=1,\ldots,N} a_i}{M\max_{i=1,\ldots,M} b_i}
 \leq \frac{\sum_{i=1}^N a_i}{\sum_{i=1}^M b_i} \leq \frac{N\max_{i=1,\ldots,N} a_i}{\max_{i=1,\ldots,M} b_i},$$which follows by direct comparison.
\end{proof}

\subsection{Proof of Lemma~\ref{lem:binom-ratio}}
\label{app:proof-binom-ratio}

\begin{proof}[Proof of Lemma~\ref{lem:binom-ratio}]
For ease of presentation, we write $k$ in place of $k'$ for this proof (since this lemma does not involve $k$). Write $k = n^{1-\bar{\delta}}$ where $\bar\delta \ge \delta$. Assume $0 < t < k$, since the cases $t = 0$ and $t = k$ can be easily verified.

In the case $t \ge ck$ (for some $0 < c < 1$ to be chosen later), we can use the simple binomial bounds $(n/k)^k \le {n \choose k} \le n^k$ to obtain
\[ \log \frac{{n \choose k-t}}{{n \choose k}} \le (k-t) \log n - k \log \frac{n}{k} = k \log k - t \log n \le \frac{t}{c}(1-\bar\delta) \log n - t \log n. \]
Thus there is a constant $c = c(\delta) < 1$ such that if $t \ge ck$ then
\[ \log \frac{{n \choose k-t}}{{n \choose k}} \le - t\bar\delta(1-\delta) \log n = -t(1-\delta) \log\left(\frac{n}{k}\right) \]
as desired.

We now treat the case $t < ck$. By Stirling's approximation, for any $n \ge 1$,
\[ \sqrt{2\pi}\, n^{n+1/2} e^{-n} \le n! \le e\, n^{n+1/2} e^{-n}. \]
By expanding ${n \choose k} = \frac{n!}{k!(n-k)!}$, this yields (for a universal constant $C$)
\begin{align*}
\log \frac{{n \choose k-t}}{{n \choose k}} &\le C + (k+1/2) \log k + (n-k+1/2) \log(n-k) \\
&\qquad- (k-t+1/2)\log(k-t) - (n-k+t+1/2)\log(n-k+t) \\
&= C + (k+1/2) \log \frac{k}{k-t} + (n-k+1/2)\log\frac{n-k}{n-k+t} + t \log \frac{k-t}{n-k+t} \\
&\le C + (k+1/2) \log \frac{k}{k-t} + t \log \frac{k-t}{n-k+t} \\
&\le C + (k + 1/2)\left(\frac{k}{k-t} - 1\right) + t \log \frac{k}{n-k} \qquad\text{using } \log(a) \le a-1 \\
&= C + (k + 1/2)\frac{t}{k-t} + t \log \frac{k}{n-k} \\
&\le C + (k + 1/2)\frac{t}{k(1-c)} + t \log \frac{2k}{n} \qquad\text{since $k \le n^{1-\delta}\le n/2$ for sufficiently large $n$} \\
&\le C + \frac{2t}{1-c} + t \log 2 - t\log \left(\frac{n}{k}\right).
\end{align*}
Since $\log(n/k) \ge \delta \log n$, this completes the proof.
\end{proof}

\subsection*{Acknowledgments}
We thank David Gamarnik for helpful comments on an earlier draft. We thank the anonymous reviewers for their helpful comments.

\bibliographystyle{alpha}
\bibliography{main}

\end{document}